\documentclass{amsart}
\usepackage{amsmath}%
\usepackage{amsthm}
\usepackage{amsfonts}%
\usepackage{amssymb}%
\usepackage{graphicx}
\usepackage{tikz-cd}
\usetikzlibrary{cd}

\usepackage{mathrsfs}

\DeclareFontFamily{OT1}{pzc}{}
\DeclareFontShape{OT1}{pzc}{m}{it}{<-> s * [1.10] pzcmi7t}{}
\DeclareMathAlphabet{\mathpzc}{OT1}{pzc}{m}{it}

\newtheorem{theorem}{Theorem}[section]

\newtheorem{corollary}[theorem]{Corollary}

\newtheorem{definition}[theorem]{Definition}

\newtheorem{lemma}[theorem]{Lemma}

\newtheorem{proposition}[theorem]{Proposition}
\newtheorem{remark}[theorem]{Remark}

\numberwithin{equation}{section}
\def\Xint#1{\mathchoice
{\XXint\displaystyle\textstyle{#1}}%
{\XXint\textstyle\scriptstyle{#1}}%
{\XXint\scriptstyle\scriptscriptstyle{#1}}%
{\XXint\scriptscriptstyle\scriptscriptstyle{#1}}%
\!\int}
\def\XXint#1#2#3{{\setbox0=\hbox{$#1{#2#3}{\int}$}
\vcenter{\hbox{$#2#3$}}\kern-.5\wd0}}

\def\dashint{\Xint-}

\newcommand{\R}{\mathbb{R}}

\newcommand{\N}{\mathbb{N}}
\newcommand{\Z}{\mathbb{Z}}

\newcommand{\F}{\mathcal{F}}
\newcommand{\Ha}{\mathcal{H}}

\newcommand{\G}{\mathscr G}

\newcommand{\spt}{\operatorname{spt}}

\newcommand{\dist}{\operatorname{dist}}
\newcommand{\diam}{\operatorname{diam}}
\newcommand{\inte}{\operatorname{int}}

\newcommand{\lip}{\operatorname{lip}}
\newcommand{\Lip}{\operatorname{Lip}}
\newcommand{\LIP}{\operatorname{LIP}}
\newcommand{\ran}{\operatorname{im}}
\newcommand{\Parcon}{\Gamma_{\operatorname{pc,c}}^k}
\newcommand{\ud}{\mathrm {d}}
\newcommand{\id}{\mathrm {id}}
\newcommand{\cl}{\mathrm {cl}}

\newcommand{\inv}{^{-1}}
\newcommand{\Rep}{\mathrm{Rep}}
\newcommand{\Poly}{{\sf Poly}}
\newcommand{\Polys}{\mathpzc{Poly}}
\newcommand{\FillVol}{\operatorname{FillVol}}

\newcommand{\loc}{\mathrm{loc}}

\newcommand{\sD}{\mathscr D}

\title[Homology of BLD-elliptic spaces]{Pull-back of metric currents and homological boundedness of BLD-elliptic spaces}

\thanks{P.P. was supported in part by the Academy of Finland project \#297258. }

\author{Pekka Pankka}
\address{P.O. Box 64 (Gustaf H\"allstr\"omin katu 2a), FI-00014 University of Helsinki, Finland}
\email{pekka.pankka@helsinki.fi}

\author{Elefterios Soultanis}
\address{SISSA, Via Bonomea 265, 34136, Trieste, Italy, and\newline
University of Fribourg, Chemin du Musee 23, CH-1700, Fribourg, Switzerland}
\email{elefterios.soultanis@gmail.com}

\date{\today}

\begin{document}
	
\begin{abstract}
Using the duality of metric currents and polylipschitz forms, we show that a BLD-mapping $f\colon X\to Y$ between oriented cohomology manifolds $X$ and $Y$ induces a pull-back operator $f^\ast \colon M_{k,\loc}(Y) \to M_{k,\loc}(X)$ between the spaces of metric $k$-currents of locally finite mass. For proper maps, the pull-back is a right-inverse (up to multiplicity) of the push-forward $f_\ast \colon M_{k,\loc}(X)\to M_{k,\loc}(Y)$.
	
As an application we obtain a non-smooth version of the cohomological boundedness theorem of Bonk and Heinonen for locally Lipschitz contractible cohomology $n$-manifolds $X$ admitting a BLD-mapping $\R^n \to X$. 
\end{abstract}

\maketitle
	
\section{Introduction}

In this article we prove a metric variant of a theorem of Bonk--Heinonen for mappings of bounded length distortion (\emph{BLD-maps} for short). 
The theorem of Bonk and Heinonen \cite[Theorem 1.1]{bonk01} states that, \emph{if $f:\R^n\to M$ is a non-constant $K$-quasiregular map into a closed and oriented Riemannian $n$-manifold $M$, then there is a constant $C(n,K)$ depending only on $n$ and $K$, such that the de Rham cohomology groups have the dimension bound}
\[
\dim H^k(M)\le C(n,K),\quad k=0,1,\ldots,n.
\]
Recently Prywes \cite{pry18} affirmed the sharp bound $C(n,K)=2^n$, conjectured in \cite{bonk01}; see Kangasniemi \cite{kan17} for a similar result in a dynamical context of uniformly quasiregular mappings. 

Recall that a continuous, open and discrete map $f\colon X\to Y$ between metric spaces is an \emph{$L$-BLD map,} for $L\ge 1$, if it satisfies the \emph{bounded length distortion estimate}
\begin{equation}\label{BLD}
\frac{1}{L} \ell(\gamma) \le \ell(f\circ \gamma) \le L \ell(\gamma)
\end{equation}
for each path $\gamma$ in $X$, where $\ell(\cdot)$ is the length of a path. We call a map $f\colon X\to Y$ simply a \emph{BLD-map} if it is $L$-BLD for some $L\ge 1$. Condition (\ref{BLD}) may be regarded as a locally non-injective variant of the bi-Lipschitz condition. Indeed, every bi-Lipschitz bijection is a BLD-map.

\begin{theorem}
	\label{thm: main}
	Suppose $X$ is a compact, geodesic, and locally Lipschitz contractible oriented cohomology manifold which admits an $L$-BLD map $f \colon \R^n\to X$. Then there exists a constant $C(n,L)$ depending only on $n$ and $L$ such that 
	\[ 
	\dim H_k(X)\le C(n,L)
	\] 
	for all $k=0,\ldots, n$.
\end{theorem}
Here the homology $H_k(X)$ is the $k^\textrm{th}$ current homology of $X$, defined using metric currents of Ambrosio and Kirchheim \cite{amb00}. For locally Lipschitz contractible spaces, the current homology $H_k(X)$ agrees with the standard singular homology of $X$. For the definition of oriented cohomology manifolds, we refer to Section \ref{sec: cohomology manifolds}.  


The key ingredient in our proof is the pull-back of metric currents. Heuristically, Theorem \ref{thm: locpull} below states the following:
\begin{itemize}
	\item[]\emph{A BLD-map between locally geodesic, oriented cohomology $n$-manifolds induces a natural pull-back operator on currents, which commutes with the boundary.}
\end{itemize}

In constructing the pull-back of currents, we use the duality of metric currents and polylipschitz forms, developed in \cite{teripekka1}, and develop a push-forward operator for polylipschitz forms under BLD-maps, analogous to a push-forward of differential forms on manifolds under quasiregular mappings discussed in \cite{Kangasniemi-Pankka}.


\subsection*{Pull-back of metric currents}
A Lipschitz map $f\colon X\to Y$ between locally compact spaces induces a push-forward
\begin{equation}\label{intropush}
f_*:M_{k}(X)\to M_{k}(Y),\quad T\mapsto T\circ(f\times\cdots\times f),
\end{equation}
between the spaces of finite mass $k$-currents $M_{k}(X)$ and $M_{k}(Y)$, respectively.

We construct the pull-back of currents by a BLD-map, using the duality of metric currents with polylipschitz forms, as an adjoint of a push-forward of polylipschitz forms. This relies on two key properties of BLD-maps between oriented cohomology manifolds.

Firstly, a BLD-map (and more generally a branched cover) $f\colon X\to Y$ between oriented cohomology manifolds admits a local index function $i_f \colon X\to \Z$. For a compactly supported Borel function $g:X\to \R$, we may define the push-forward $f_\#g:Y\to \R$ of $g$ by $f$ as the function 
\begin{equation}\label{eq: pushforward}
y\mapsto \sum_{x\in f\inv (y)}i_f(x)g(x).
\end{equation}

Secondly, if $X$ and $Y$ are locally geodesic, the lower estimate in \eqref{BLD} allows us to obtain a Lipschitz estimate for the push-forward of Lipschitz functions; see Lemma \ref{locpush}.

Using these properties we define a push-forward operator \[
f_\# \colon \Gamma^k_c(X) \to \Gamma^k_{\mathrm{pc,c}}(Y)
\]
from the space of polylipschitz forms to the space of partition continuous polylipschitz forms $\Gamma^k_{\mathrm{pc,c}}(Y)$. We refer to Section \ref{sec:polylip} for polylipschitz forms and Section \ref{sec:push-forward_forms} for the details on the push-forward. The pull-back $f^\ast T$ of a current $T\in M_{k,\loc}(Y)$ of locally finite mass is then obtained by the formula
\begin{equation}
\label{eq:pull-back_intro}
f^\ast T(\pi_0,\ldots, \pi_k) = \widehat T (f_\#[\pi_0,\ldots, \pi_k]),
\end{equation}
for $(\pi_0,\ldots, \pi_k)\in \LIP_c(X)\times\LIP_\infty(X)^k$; see \cite{lan11} for currents of locally finite mass on locally compact spaces. Here $\widehat T$ is the extension of $T$ to the space of piecewise continuous polylipschitz forms; see \cite[Theorem 1.3]{teripekka1}

\begin{theorem}\label{thm: locpull}
	Let $f:X\to Y$ be an $L$-BLD map between locally geodesic oriented cohomology manifolds, and let $k\in \N$. Then there is a natural and weakly sequentially continuous linear map 
	\[ 
	f^\ast: M_{k,loc}(Y)\to M_{k,loc}(X) 
	\] 
	having the following properties:
	\begin{enumerate}
		\item for a $k$-current $T\in M_{k,loc}(Y)$ and a precompact Borel set $E\subset X$, 
		\[
		f_\ast((f^\ast T)\lfloor E)=T\lfloor {f_\# \chi_E};
		\]
		\label{eq:locpull1}
		\item $\partial \circ f^* = f^* \circ \partial \colon N_{k,\loc}(Y) \to N_{k-1, \loc}(X)$; and
		\item for each $T\in M_{k,\loc}(Y)$, 
		\[
		\frac{1}{L^k} f^\ast\|T\|\le \|f^\ast T\|\le L^k f^\ast\|T\|,
		\]
		where $\|\cdot\|$ is the mass measure of a current.
	\end{enumerate}
\end{theorem}

The naturality of the map $f^* \colon M_{k,\loc}(Y) \to M_{k,\loc}(X)$ in the statement refers to the typical functorial properties of the pull-back, that is, given BLD-maps $f\colon X\to Y$ and $g\colon Y \to Z$ between locally geodesic oriented cohomology manifolds, we have the composition rule $f^* \circ g^*  = (g\circ f)^*$.

Property \eqref{eq:locpull1} in Theorem \ref{thm: locpull} characterizes the pull-back in the following sense.
\begin{theorem}\label{unique}
	Let $f:X\to Y$ be a BLD-map between geodesic oriented cohomology manifolds and $T\in M_{k,loc}(Y)$. Suppose $S\in M_{k,loc}(X)$ is such that \[ f_\ast(S\lfloor E)=T\lfloor f_\#\chi_E \] for all precompact Borel sets $E\subset X$. Then \[ S=f^\ast T. \]
\end{theorem}

Properties (2) and (3) in Theorem \ref{thm: locpull} immediately imply the stability of normal currents under the pull-back.

\begin{corollary}
	Let $f:X\to Y$ be a BLD map between locally geodesic oriented cohomology manifolds, and let $k\in \N$. Then the pull-back $f^\ast \colon M_{k,\loc}(Y) \to M_{k,\loc}(X)$ restricts to an operator 
	\[ 
	f^\ast: N_{k,\loc}(Y)\to N_{k,\loc}(X).
	\]
\end{corollary}

For proper BLD-maps, we may characterize the pull-back as a right inverse of the push-forward as follows. Recall that a map $f:X\to Y$ is \emph{proper} if $f\inv (K)\subset X$ is compact whenever $K\subset Y$ is compact and that proper maps between oriented cohomology manifolds have a global degree $\deg f\in\Z$.

\begin{corollary}
	\label{thm:pull-back}
	Let $f\colon X\to Y$ be a proper $L$-BLD-map between locally geodesic, oriented cohomology manifold. Then the pull-back $f^* \colon M_k(Y) \to M_k(X)$ has the following properties:
	\begin{enumerate}
		\item the composition $f_* \circ f^* \colon M_{k,\loc}(Y) \to M_{k,\loc}(Y)$ satisfies $f_\ast \circ f^\ast = (\deg f) \id$; 
		\item the pull-back $f^\ast$ commutes with the boundary, i.e. $\partial f^\ast T=f^\ast(\partial T)$ for $T\in M_{k,\loc}(X)$; and
		\item for each $T\in M_k(X)$, 
		\[
		\frac{1}{L^{k}}(\deg f)M(T)\le M(f^\ast T)\le L^k(\deg f)M(T),
		\]
		where $M(T)=\|T\|(X)$ is the mass of a finite mass current $T\in M_k(X)$.
	\end{enumerate}
	Moreover,  the pull-back operator $f^\ast$ restricts to a natural operator \[ f^\ast: N_{k}(Y)\to N_{k}(X). \]
\end{corollary}

\subsection*{Homological boundedness of BLD-elliptic spaces}

As stated in the beginning of the introduction, our motivation to consider the pull-back operator for currents comes from the question of cohomological boundedness of oriented cohomology manifolds admitting a BLD-mapping from the Euclidean space, or \emph{BLD-elliptic spaces}. 
This terminology is an adaptation on the notion of \emph{quasiregular ellipticity} for closed $n$-manifolds admitting a quasiregular map from $\R^n$ introduced by Bonk and Heinonen \cite{bonk01}, which in turn is an adaptation of \emph{ellipticity of a manifold}, introduced by Gromov \cite{gro07}.



Our proof follows the strategy of Bonk and Heinonen in \cite{bonk01}. Instead of considering pull-backs of $p$-harmonic forms as in \cite{bonk01}, we consider pull-backs of metric currents. For this reason, the first step of the argument is to show that the singular homology is isomorphic to the homology of normal metric currents on $X$. It is an interesting question whether it is possible to develop the method of Prywes in this context.

A comment on the assumptions in Theorem \ref{thm: main} is in order. We assume that the target space $X$ is \emph{locally Lipschitz contractible}, that is, we assume that for every point $x\in X$ and a neighborhood $U$ of $x$ there is a neighborhood $V\subset U$ of $x$ and a Lipschitz map $h \colon V\times [0,1]\to U$ for which $h_0$ is the inclusion $V\hookrightarrow U$ and $h_1$ is a constant map. Clearly, Riemannian manifolds in the theorem of Bonk and Heinonen are locally Lipschitz contractible. In the proof of Theorem \ref{thm: main} this assumption yields an a priori finite dimensionality for the current homology $H_*(X)$, which in turn allows us to obtain filling inequality (Proposition \ref{main2}) for normal currents on $X$.

A locally geodesic, locally Lipschitz contractible, and orientable cohomology $n$-manifold admitting a BLD-map from $\R^n$ is a generalized manifold of type A in the terminology of Heinonen and Rickman; see \cite[Definition 5.1]{hei02} for the definition. Indeed, local Lipschitz contractibility implies local linear contractibility, where no Lipschitz requirement is placed on the contracting homotopy. Local Ahlfors-regularity follows from the work of Heinonen--Rickman \cite{hei02} and is discussed in Section \ref{sec:BLDrn}; see Remark \ref{rem:bound}. The last nontrivial condition, the local bilipschitz embeddability, follows from the work of Almgren; see \cite{alm2000} and De Lellis--Spadaro \cite{del11}. We give the details in Appendix \ref{appendix}; see Theorem \ref{lem:almgren}.

\medskip\noindent This article is organized as follows. In Sections \ref{sec:preliminaries} and \ref{sec:BLD}, we discuss preliminaries on metric currents and polylipschitz forms, and BLD-mappings, respectively. Section \ref{sec:pull-back} is devoted to the pull-back of metric currents under BLD-maps and we prove Theorems \ref{thm: locpull} and \ref{unique} in this section. In Section \ref{sec:equidistribution} we prove equidistribution of the pull-back currents under BLD-maps and in Section \ref{sec:current_homology} we discuss current homology and prove Theorem \ref{thm: main}. The article is concluded with an appendix on bilipschitz embeddability of BLD-elliptic spaces into Euclidean spaces.

\medskip\noindent {\bf Acknowledgements} We thank Rami Luisto and Stefan Wenger for discussions on the topics of the manuscript.

\section{Metric currents and polylipschitz forms}
\label{sec:preliminaries}

In this section, we recall first basic notions from the Ambrosio--Kirchheim theory of metric currents \cite{amb00} and then briefly discuss the construction of polylipschitz forms introduced in \cite{teripekka1}.
\subsection{Metric currents}

Let $X$ be a locally compact metric space. A function $f:X\to \R$ is \emph{Lipschitz}, if
\[
\Lip(f):=\sup_{x\ne y}\frac{|f(x)-f(y)|}{d(x,y)}<\infty.
\]
We denote by $\LIP_\infty(X)$ and $\LIP_c(X)$ the vector spaces of bounded Lipschitz functions and Lipschitz functions with compact support, respectively. We equip $\LIP_\infty(X)$ and $\LIP_c(X)$ with locally convex vector topologies such that
\begin{itemize}
	\item[(1)] $f_n\to f$ in $\LIP_\infty(X)$ if $f_n\to f$ pointwise and $\sup_n\LIP(f_n)<\infty$; and
	\item[(2)] $f_n\to f$ in $\LIP_c(X)$ if there is a compact set $K\subset X$ for which $\spt(f_n)\subset K$ for all $n\in\N$, and $f_n\to f$ in $\LIP_\infty(X)$.
\end{itemize}
See \cite{lan11} for more details. Given $k\ge 0$, let
\[
\sD^k(X):=\LIP_c(X)\times \LIP_\infty(X)^k
\]
be equipped with the product topology. A $(k+1)$-linear map $\displaystyle T:\sD^k(X)\to \R$ is a \emph{metric $k$-current on $X$} if
\begin{itemize}
	\item[(1)] $\displaystyle \lim_{n\to\infty}T(\pi_n)=T(\pi)$ whenever $\pi_n\to \pi$ in $\sD^k(X)$, and 
	\item[(2)] $T(\pi_0,\ldots,\pi_k)=0$ whenever, for some $j=1,\ldots,k$, $\pi_j$ is constant in a neighbourhood of $\spt \pi_0$.
\end{itemize}
The vector space of metric $k$-currents is denoted $\sD_k(X)$.

\subsubsection*{Boundary and restriction} For $k\ge 1$, the \emph{boundary operator} $\partial:\sD_k(X)\to sD_{k-1}(X)$ is defined by
\[
\partial T(\pi_0,\ldots,\pi_{k-1}):=T(\sigma,\pi_0,\ldots,\pi_{k-1})
\]
for any $\sigma\in \LIP_c(X)$ with $\sigma|_{\spt \pi_0}\equiv 1$. It follows from the locality condition (2) that the boundary is well defined.

Given $\alpha=(\alpha_0,\ldots,\alpha_m)\in \sD^m(X)$ and $T\in \sD_k(X)$ with $k\ge m$, we may define the \emph{restriction of $T$ by $\alpha$} as the current
\[
T\lfloor\alpha\in \sD_{k-m}(X),\quad (\pi_0,\ldots,\pi_{k-m})\mapsto T(\alpha_0\pi_0,\alpha_1,\ldots,\alpha_m,\pi_1,\ldots,\pi_{k-m}).
\]

\subsubsection*{Mass} A $k$-current $T\in \sD_k(X)$ is said to have \emph{locally finite mass}, if there is a Radon measure $\mu$ on $X$ satisfying
\begin{equation}\label{mass}
|T(\pi_0,\ldots,\pi_k)|\le \Lip(\pi_1)\cdots\Lip(\pi_k)\int_X|\pi_0|\ud\mu,\quad (\pi_0,\ldots,\pi_k)\in \sD^k(X)
\end{equation}
for every $(\pi_0,\ldots,\pi_k)\in \sD^k(X)$. If $T\in \sD_k(X)$ has locally finite mass, it admits a \emph{mass measure}, denoted $\|T\|$, a Radon measure on $X$ that is minimal with respect to satisfying (\ref{mass}). If $\|T\|(X)<\infty$, we say that $T$ has \emph{finite mass}. The space of $k$-currents of locally finite mass is denoted by $M_{k,\loc}(X)$, and the space of $k$-currents of finite mass $M_k(X)$.

\subsubsection*{Normal currents} A $k$-current $T\in \sD_k(X)$ is called \emph{locally normal}, if $T\in M_{k,\loc}(X)$ and $\partial T\in M_{k-1,\loc}(X)$, and \emph{normal} if $T\in M_k(X)$ and $\partial T\in M_{k-1}(X)$. The \emph{normal mass} $N \colon N_k(X)\to [0,\infty)$, 
\[
T \mapsto \|T\|(X) + \|\partial T\|(X),
\]
is a norm on $N_k(X)$ and the normed space $(N_k(X), N)$ is a Banach space; see \cite[Proposition 4.2]{lan11}. For $k=0$, the norm $N$ is the total variation norm.


\subsubsection*{Flat norm} Let $E\subset X$ be a Borel set and let $\F_E \colon N_{k,\loc}(X) \to [0,\infty]$ be the function
\[
T\mapsto \inf \{ \|T-\partial A\|(E)+\|A\|(E): A\in N_{k+1,\loc}(X) \}.
\]
For each non-empty Borel set $E$, $\F_E$ is a seminorm and  $\F := \F_X$ is a norm on $N_k(X)$, called the \emph{flat norm of $N_{k}(X)$}.

We recall that, for each $T\in N_{k,\loc}(X)$ and a Borel set $E\subset X$,  
\[
\F_E(\partial T) \le \F_E(T) \le \|T\|(E).
\]

We record standard properties of the flat norm of a restriction of a current as a lemma. 
\begin{lemma}\label{flatprod}
	Let $T\in N_{k,\loc}(X)$ and $E\subset X$ a Borel set. Then 
	\[
	\F_E(T\lfloor\eta)\le (\|\eta\|_\infty+\Lip\eta)\F_E(T)
	\] 
	for all Lipschitz functions $\eta\in\LIP_\infty(X)$. Moreover, if $\eta\in\LIP_\infty(X)$ satisfies $\eta|_E\equiv 1$, then  
	\[ 
	\F_E(T)=\F_E(T\lfloor\eta). 
	\]
\end{lemma}

\begin{proof}
	Let $A\in N_{k,\loc}(X)$ and $B\in N_{k+1,\loc}(X)$ satisfy $T=A+\partial B$. Then, for any $\eta\in \LIP_c(X)$ we have
	\[\begin{split}
	T\lfloor \eta=A\lfloor \eta+(\partial B)\lfloor\eta=A\lfloor\eta+B\lfloor (1,\eta)+\partial(B\lfloor\eta).
	\end{split}\]
	Thus, by (\ref{mass}), we have
	\[
	\begin{split}
	\F_E(T\lfloor\eta)\le &\|A\lfloor\eta+B\lfloor (1,\eta)\|(E)+\|B\lfloor\eta\|(E)\\
	\le &\|A\lfloor \eta\|(E)+\|B\lfloor(1,\eta)\|(E)+\|B\lfloor\eta\|(E)\\
	\le &\|\eta\|_\infty\|A\|(E)+(\Lip\eta)\|B\|(E)+\|\eta\|_\infty\|B\|(E)\\
	\le &(\|\eta\|_\infty+\Lip\eta)\big(\|A\|(E)+\|B\|(E)\big).
	\end{split}
	\]
	The first claim follows.
	
	For the second claim, note that
	\[
	\F_E(T)=\inf\{ \|T-\partial\|(E)+\|A\|(E): A\in N_{k+1}(X) \},
	\]
	and that 
	\[ 
	\|T\lfloor(1-\eta)\|(E)=0 
	\] 
	if $\eta|_E\equiv 1$. For any $A\in N_{k+1}(X)$ we have
	\[
	\|T\lfloor\eta-\partial A\|(E)-\|T-T\lfloor\eta\|(E)\le \|T-\partial A\|(E)\le \|T-T\lfloor\eta\|(E)+\|T\lfloor\eta-\partial A\|(E).
	\]
	Thus 
	\[ 
	\|T-\partial A\|(E)=\|T\lfloor\eta-\partial A\|(E), 
	\] 
	and the second claim follows.
\end{proof}

The following compactness result for the flat norm provides a crucial tool in the proof of homological boundedness. This result is used for currents in $\R^n$ and it is an immediate consequence of \cite[Corollary 7.3]{fed60} and the weak compactness of normal currents \cite[Theorem 5.4]{lan11}. For an analogous compactness result in compact metric spaces, see \cite{dep17}.

\begin{theorem}
	\label{compactness}
	Let $A\subset \R^n$ be a compact subset and $\lambda\ge 0$. Then the set 
	\[ 
	N_k(A,\lambda) = \{T\in N_k(\R^n): \spt T\subset A \text{ and } N(T)\le\lambda\}
	\] 
	is compact in the flat norm $\F_A$, in the sense that every sequence $(T_i)$ in $N_k(A,\lambda)$ has a subsequence $(T_{i_k})$ and $T\in N_k(A,\lambda)$ such that $$\lim_{k\to\infty}\F_A(T_{i_k}-T)= 0.$$
\end{theorem}

\subsection{Polylipschitz forms}\label{sec:polylip}
\subsubsection*{Polylipschitz functions} Let $k\in\N$ and $X$ be a metric space. Given functions $f_0,\ldots,f_k:X\to\R$, denote by $f_0\otimes\cdots\otimes f_k:X^{k+1}\to\R$ the function
\[
(x_0,\ldots,x_k)\mapsto f_0(x_0)\cdots f_k(x_k).
\]
Note that if each $f_j$ is Lipschitz and bounded, then $f_1\otimes\cdots\otimes f_k$ is Lipschitz and bounded on $X^{k+1}$ (here we endow $X^{k+1}$ with the Euclidean product metric). The norm $L(f)$ on $\LIP_\infty(X)$, given by
\[
L(f):=\max\{\|f\|_\infty,\Lip(f) \}\quad \text{for each }f\in \LIP_\infty(X),
\]
makes $\LIP_\infty(X)$ into a Banach space.

Consider the algebraic tensor product $\LIP_\infty(X)^{\otimes(k+1)}$. The \emph{projective tensor norm} on $\LIP_\infty(X)^{\otimes(k+1)}$ is given by 
\[
L_k(\pi)=\inf\left\{\sum_j^mL(\pi_0^j)\cdots L(\pi_k^j): \pi=\sum_j^m\pi_0^j\otimes\cdots\otimes\pi_k^j \right\},\quad \pi\in\LIP_\infty(X)^{\otimes(k+1)}.
\]
The completion of $\LIP_\infty(X)^{\otimes(k+1)}$ with respect to the projective tensor norm is called the \emph{(completed) projective tensor product} and denoted $\LIP_\infty(X)^{\hat\otimes_\pi(k+1)}$.

The projective tensor product has the following \emph{universal property} which characterizes it up to isometric isomorphism in the category of Banach spaces: Let $B$ be a Banach space and $A:\LIP_\infty(X)^{k+1}\to B$ a continuous $(k+1)$-linear map. Then there exists a unique \emph{continuous linear map} $ \overline A:\LIP_\infty(X)^{\hat\otimes_\pi(k+1)}\to B$ satisfying
\begin{equation}\label{diagram:proj}
A = \overline{A} \circ \jmath,
\end{equation}
where $\jmath:\LIP_\infty(X)^{k+1}\to \LIP_\infty(X)^{\hat\otimes_\pi(k+1)}$ is the continuous $(k+1)$-linear map $(\pi_0,\ldots,\pi_k)\mapsto \pi_0\otimes\cdots\otimes\pi_k$.
In particular, the  map
\[
A:\LIP_\infty(X)^{k+1}\to \LIP_\infty(X^{k+1}),\quad (\pi_0,\ldots,\pi_k)\mapsto \pi_0\otimes\cdots\otimes\pi_k,
\]
extends to a continuous linear map
\[
\bar A:\LIP_\infty(X)^{\hat\otimes_\pi(k+1)}\to \LIP_\infty(X^{k+1}).
\]
we identify the projective tensor product with the image of this map in $\LIP_\infty(X^{k+1})$. 
\begin{definition}
	Let $X$ be a metric space and $k\in\N$. A function $\pi:X^{k+1}\to\R$ is a \emph{$k$-polylipschitz function on $X$} if there are bounded Lipschitz functions $\pi_0^j,\ldots\pi_k^j\in\LIP_\infty(X)$, $j=0,1,\ldots$, satisfying
	\begin{equation}\label{polylip1}
	\sum_j^\infty L(\pi_0^j)\cdots L(\pi_k^j)<\infty,
	\end{equation}
	and
	\begin{equation}\label{polylip2}
	\pi=\sum_j^\infty\pi_0^j\otimes\cdots\otimes\pi_k^j.
	\end{equation}
\end{definition}
In other words a polylipschitz function is an element of the completed  projective tensor product under the identification explained above.

\bigskip\subsubsection*{Polylipschitz forms}
Fix a metric space $X$. Given open sets $U\subset V\subset X$ we denote by
\[
\rho_{U,V}:\Poly^k(V)\mapsto \Poly^k(U),\quad \pi\mapsto \pi|_{U^{k+1}},
\]
the restriction map. The collection
\[
\{\Poly^k(U), \rho_{U,V} \}
\]
ranging over all open sets $U\subset V\subset X$ is known as the \emph{(polylipschitz) presheave over $X$}. Given $x\in X$ and two polylipschitz functions $\pi\in \Poly^k(U),\pi'\in\Poly^k(U')$ defined on open neighbourhoods $U$ and $U'$ of $x$, respectively, we say that $\pi$ and $\pi'$ are equivalent, denoted $\pi\sim\pi'$, if there is a neighbourhood $W\subset U\cap U'$ such that
\[
\rho_{W,U}(\pi)=\rho_{W,U'}(\pi').
\]
The equivalence class $[\pi]_x$ of a polylipschitz $\pi\in\Poly^k(U)$ defined on a neighbourhood $U$ of $x$ is called the \emph{germ of $\pi$ on $x$}. 

The \emph{\'etal\'e space} $\Polys^k(X)$ consists over all such equivalence classes. There is a natural projection map
\[
q:\Polys^k(X)\to X,\quad [\pi]_x\mapsto x.
\]
For each $x\in X$, the set
\[
q\inv(x)=:\Polys_x^k(X)
\]
is called the \emph{stalk of $\Polys^k(X)$ at $x$}, and it is a real vector space.

A \emph{$k$-polylipschitz section on $X$} is a section of $\Polys^k(X)$, i.e. a map $\omega:X\to \Polys^k(X)$ satisfying $q\circ\omega=\id_X$. We denote the space of $k$-polylipschitz sections on $X$ by $\G^k(X)$. The support of a $k$-polylipschitz section $\omega\in \G^k(X)$ is the set
\[
\spt\omega=\cl \{ x\in X: \omega(x)\ne 0 \}.
\]

The space $\Polys^k(X)$ can be equipped with the \emph{\'etal\'e topology} which makes $q$ into a local homeomorphism. See \cite[Section 5.6]{war83} for the details. Note that $\Polys^k(X)$ is usually a rather pathological space; for example it is rarely Hausdorff. Instead of describing the topology, we describe what continuity of sections means: a section $\omega$ is continuous if there is there is a locally finite open cover $\mathscr U$ of $X$ and a collection $\{ \pi_U \}_{U\in\mathscr U}$, where $\pi_U\in \Poly^k(U)$, such that $[\pi_U]_x=\omega(x)$ for all $U\in\mathscr U$ and $x\in U$, and the collection $\{\pi_U \}_{U\in\mathscr U}$ satisfies the \emph{overlap condition}
\begin{equation}\label{overlap}
\rho_{U\cap V,U}(\pi_U)=\rho_{U\cap V,V}(\pi_V)
\end{equation}
whenever $U,V\in \mathscr U$ and $U\cap V\ne \varnothing$.

Conversely, any collection $\{ \pi_U \}_{U\in\mathscr U}$ satisfying (\ref{overlap}) defines a continuous section $\omega$ of $\Polys^k(X)$ by setting
\[
\omega(x)=[\pi_U]_x,\quad \textrm{whenever $U\in\mathscr U$ and } x\in U.
\]
\begin{definition}\label{polylipform}
	Let $X$ be a metric space, and $k\in\N$. A \emph{$k$-polylipschitz form} on $X$ is a continuous section of $\Polys^k(X)$.
\end{definition}
The space of $k$-polylipschitz forms on $X$ is denoted by $\Gamma^k(X)$, and $\Gamma_c^k(X)$ denotes the set of polylipschitz forms whose support is compact.

\subsubsection*{Piecewise continuous polylipschitz forms} 

Given any set $B\subset X$, the restriction operators $\rho_{U\cap B,V}:\Poly^k(V)\to\Poly^k(U\cap B)$, for $U\subset V\subset X$, form a presheaf homomorphism, giving rise to a restriction homomorphism $\rho_B:\G^k(X)\to\G^k(B)$, where $B$ is considered as a metric space with the restricted metric from $X$. We denote $\rho_B(\omega)=:\omega|_B$ for $\omega\in \G^k(X)$.
\begin{definition}
	A $k$-polylipschitz section $\omega\in \G^k(X)$ is called \emph{$\mathcal E$-continuous}, where $\mathcal E$ is a countable Borel partition of $X$, if $\omega|_B\in \Gamma^k(B)$ for every $B\in \mathcal E$.
	
	A polylipschitz section is \emph{partition-continuous} if it is $\mathcal E$-continuous for some countable Borel partition $\mathcal E$ of $X$.
\end{definition}
We denote by $\Gamma^k_{\mathrm{pc}}(X)$ the space of partition-continuous polylipschitz sections, and by $\Parcon(X)$ those elements of $\Gamma^k_{\mathrm{pc}}(X)$ which have compact support. Clearly $\Gamma_c^k(X)\subset \Parcon(X)$.

\subsubsection*{Exterior derivative and cup-product} We refer to \cite[Section 4.5]{teripekka1} for further details. Following the construction of Alexander--Spanier cohomology we introduce the linear map $\displaystyle d=d_X^k:\Poly^k(X)\to \Poly^{k+1}(X)$ by 
\[
d\pi(x_0,\ldots,x_{k+1})=\sum_{j=0}^{k+1}(-1)^j\pi(x_0,\ldots,\hat{x_j},\ldots,x_{k+1})
\]
for $\pi\in\Poly^k(X)$ and $x_0,\ldots,x_{k+1}\in X$. This map satisfies $d\circ d=0$. The presheaf homomorphism $\{ d^k_U \}_{U}$ induces homomorphism $d:\G^k(X)\to \G^{k+1}(X)$ that restricts to
\begin{align*}
&d:\Gamma_c^k(X)\to \Gamma_c^{k+1}(X)\textrm{ and}\\
&d:\Parcon(X)\to \Gamma_{\operatorname{pc},c}^{k+1}(X).
\end{align*}

The \emph{cup-product} is a bilinear map $\smile:\Parcon(X)\times\Gamma^m_{\operatorname{pc},c}(X)\to \Gamma^{k+m}_{\operatorname{pc},c}(X)$, defined in the same manner starting from the bilinear map
\[
\smile:\Poly^k(X)\times\Poly^m(X)\to \Poly^{k+m}(X)
\]
given by
\[
\alpha\smile\beta(x_0,\ldots,x_{k+m})=\alpha(x_0,\ldots,x_k)\beta(x_0,x_{k+1},\ldots,x_{k+m})
\]
for $\alpha\in \Poly^k(X)$, $\beta\in \Poly^m(X)$ and $x_0,\ldots,x_{k+m}\in X$. 
Note that the cup product restricts to a bilinear map $\smile:\Gamma_c^k(X)\times\Gamma_c^m(X)\to \Gamma^{k+m}_c(X)$.

\subsection{Duality of metric currents and polylipschitz forms}

We refer to \cite[Sections 4.2, 5.1 and 6]{teripekka1} for the notions of convergence of sequences of polylipschitz functions, polylipschitz forms, and partition-continuous polylipschitz forms, respectively. Note that the exterior derivative $d:\Parcon(X)\to \Gamma^{k+1}_{\operatorname{pc},c}(X)$ is sequentially continuous, cf. \cite[Proposition 6.8]{teripekka1}.

Recall the natural embedding
\[
\imath:\sD^k(X)\hookrightarrow \Gamma_c^k(X),\quad \imath(\pi_0,\ldots,\pi_k)(x)=[\pi_0\otimes\cdots\otimes\pi_k]_x
\]
for $(\pi_0,\ldots,\pi_k)\in\sD^k(X)$ and $x\in X$. We slightly abuse notation by using the symbol $\imath$ also for the embedding $\sD^k(X)\hookrightarrow \Parcon(X)$.

\begin{theorem}\cite[Theorems 1.1 and 1.3]{teripekka1}\label{ext}
Let $X$ be a locally compact space, $k\ge 0$, and $T\in M_{k,\loc}(X)$. Then there is a unique sequentially continuous linear functional $\widehat T:\Parcon(X)\to\R$ such that
\[
T=\widehat T\circ\imath.
\]
Moreover, if $T\in N_{k+1,\loc}(X)$, then we have
\[
\widehat{\partial T}(\omega)=\widehat T(d\omega)
\]
for every $\omega\in\Parcon(X)$.
\end{theorem}

Extensions of currents of finite mass also satisfy natural integrability bounds. Given $\pi\in \Poly^k(X)$ and $V\subset X$, define a variant of the projective norm $L_k(\cdot)$ as follows:
\[
\Lip_k(\pi;V)=\inf\left\{ \sum_j^\infty\|\pi_0^j|_V\|_\infty\Lip(\pi_1^j|_V)\cdots\Lip(\pi_k^j|_V):\ \pi=\sum_j^\infty\pi_0^j\otimes\cdots\otimes\pi_k^j \right\}.
\]

Define the \emph{pointwise norm} $\|\omega\|_x$ of $\omega\in\Gamma^k(X)$ at $x\in X$, by
\[
\|\omega\|_x:=\inf\{\Lip_k(\pi;B(x,r)):r>0\}=\lim_{r\to 0}\Lip_k(\pi;B(x,r))
\]
for any $\pi$ such that $[\pi]_x=\omega(x)$. The map $x\mapsto \|\omega\|_x$ is easily seen to be upper semicontinuous, cf. \cite[Section 6.1]{teripekka1}.

\begin{proposition}\cite[Theorem 1.2]{teripekka1}
	Suppose $T\in M_{k,\loc}(X)$, and denote by $\|T\|$ the mass measure of $T$. Then
	\[
	|\widehat T(\omega)|\le \int_X\|\omega\|_x\ud \|T\|(x),\quad \omega\in \Gamma^k(X).
	\]
\end{proposition}

\begin{remark}
	In the forthcoming sections we do not distinguish a metric current $T\in M_{k,\loc}(X)$ from the extension$\widehat T$  provided by Theorem \ref{ext}. We will consider metric currents as acting on $\sD^k(X)$, $\Poly_c^k(X), \Gamma_c^k(X)$, or $\Parcon(X)$ interchangeably and without mentioning it explicitly.
\end{remark}
	
\section{Preliminaries on BLD-maps}
\label{sec:BLD}

\subsection{Branched covers}
\label{sec:branched_covers}

A continuous mapping $f\colon X\to Y$ between metric spaces is a \emph{branched cover} if $f$ is discrete and open; recall that the map $f$ is \emph{discrete} if the pre-image $f^{-1}(y)$ of a point $y\in Y$ is a discrete set, and $f$ is \emph{open} if the image $fU$ of an open set $U\subset X$ is open. A continuous map $f\colon X\to Y$ is \emph{proper} if the pre-image $f^{-1}E$ of a compact set $E\subset Y$ is compact. In what follows, all mappings between metric spaces are continuous unless otherwise stated.

A pre-compact domain $U\subset X$ is a \emph{normal domain of $f\colon X\to Y$} if $\partial fU = f\partial U$. Further, if $x\in U$ has the property that $f^{-1}f(x) \cap \overline{U} = \{x\}$, we say that $U$ is a \emph{normal neighborhood of $x$ (with respect to $f$)}.

We recall that, given a branched cover $f\colon X\to Y$ and a normal domain $U\subset X$ of $f$, the restriction $f|_{U} \colon U \to fU$ is a proper map; see e.g.~Rickman \cite{ric93} and V\"ais\"al\"a \cite{vai66}. 

Let $f:X\to Y$ be a branched cover between locally compact spaces. For $x\in X$ and $r>0$, we denote by $U_f(x,r)$ the connected component of $f\inv B(f(x),r)$ containing $x$. When the map $f$ is clear from the context we omit the subscript and write $U(x,r)$ in place of $U_f(x,r)$. The following lemma is extensively used throughout the paper. It follows from \cite[Lemma 2.1]{lui17}; see also \cite[Lemma I.4.9]{ric93} and  \cite[Lemma 5.1.]{vai66}.

\begin{lemma}\label{spread}
	Let $f:X\to Y$ be a branched cover between locally compact spaces $X$ and $Y$. Then the following conditions hold.
	\begin{itemize}
		\item[(a)] For every $x\in X$, there exists a radius $r_x>0$, for which $U(x,r)$ is a normal domain of $x$ for every $r<r_x$. Furthermore, given a compact set $K\subset X$ and $y\in f(K)$, there exists $r_y>0$ so that $U(x,r)$ is a normal neighborhood for $x$, for every $x\in f\inv(y)\cap K$ and $r<r_y$.
		\item[(b)] If $U'\subset U\subset X$ are normal domains of $x\in X$, then $U\cap f\inv(f(U'))=U'$. 
	\end{itemize}
\end{lemma}

\begin{remark}\label{proper}
	It follows that, if $f:X\to Y$ is a proper branched cover, then, for every $x\in X$, there is a radius $r_0>0$ for which \[ f\inv B_r(f(x))=\bigcup_{x'\in f\inv(f(x))}U(x',r), \] where $U(x',r)$ is a normal neighborhood of $x'\in f\inv(f(x))$, for each $x'\in f\inv(f(x))$.
\end{remark}

\subsection{Oriented cohomology manifolds}\label{sec: cohomology manifolds} Following \cite{hei02} we say that a separable and locally compact space $X$ is an \emph{oriented cohomology $n$-manifold} if 
\begin{itemize}
	\item[(a)] $X$ has finite covering dimension,
	\item[(b)] $H^k_c(U;\Z)=0$ for each open set $U\subset X$ for $k\ge n+1$,
	\item[(c)] $H_c^n(X;\Z)\simeq \Z$, and
	\item[(d)] each point $x$ and its neighborhood $U$ contains a neighborhood $V\subset U$ of $x$ for which
	\begin{align*}
		H^{k}_c(V;\Z)=\left\{
		\begin{array}{ll}
			0&,\ k={n-1}\\
			\Z&,\ k=n
		\end{array}
		\right.
	\end{align*}
	and the standard  homomorphism $$ H^n_c(W;\Z)\to H^n_c(V;\Z)$$ is a surjection for any neighborhood $W$ of $x$ contained in $V$.
\end{itemize}
The notation $H^\ast_c( - ;\Z)$ above refers to the compactly supported Alexander-Spanier cohomology  with integer coefficients. We refer to \cite[Definition 1.1]{hei02} and the ensuing discussion for more details. Here we only mention that a more widely used notion of cohomology manifolds requires \emph{all} local cohomology groups of dimension $0<k<n$ to vanish, see e.g. \cite[Definition 6.17]{bredon97}.

\subsection{Global and local degree} 
Let $X$ and $Y$ be oriented cohomology manifolds of the same dimension $n\in\N$ and fix \emph{orientations} $c_X$ and $c_Y$ of $X$ and $Y$, i.e. generators $c_X$ and $c_Y$ of $H^n_c(X;\Z)$ and $H^n_c(Y;\Z)$, respectively. For open sets $U\subset X$ and $V\subset Y$ we have local orientations given by $c_U=\iota_{UX}^\ast c_X$ and $c_V=\iota_{VY}^\ast c_Y$, where $\iota_{UX}:U\hookrightarrow X$ and $\iota_{VY}:V\hookrightarrow Y$ are inclusions. As described in \cite{ric93, vai66, hei02}, continuous maps $X\to Y$ admit a local degree in the following sense. Here we follow the presentation in \cite{hei02}. 

Given a precompact domain $U\subset X$, the \emph{local degree $\mu_f(U,y)\in\Z$ with respect to a point} $y\in Y\setminus f(\partial U)$ and domain $U$ is
\begin{enumerate}
	\item 0 if $y\notin \inte f(U)$, and otherwise 
	\item the unique integer $\lambda\in \Z$ for which the pull-back homomorphism \[(f|_U)^\ast: H^n_c(V)\to H^n_c(U)\textrm{ satisfies } f^\ast[c_{V}]=\lambda[c_U],  \] where $V$ is the component of $\inte f(U)\setminus f(\partial U)$ containing $y$; note that $y\in \inte f(U)$ in this case. 
\end{enumerate}
Then $\mu_f(U,y)$ is constant in each component of $\inte f(U)\setminus f(\partial U)$.

If $f:X\to Y$ is a proper map, it admits a \emph{global degree} $\deg f$, which is the unique integer $\lambda\in\Z$ for which the pull-back $f^* \colon H^n_c(Y)\to H^n_c(X)$ in cohomology satisfies $$ f^\ast c_Y=\lambda\ c_X.$$

\bigskip\noindent A standard property of the local degree is that, for precompact domains $V\subset U$ and a point $y\in Y$ satisfying $y\notin f(\partial V)\cup f(\partial U)$ and $f\inv(y)\cap U\subset V$, we have \[ \mu_f(V,y)=\mu_f(U,y). \] This immediately yields a summation formula
\begin{equation}
	\label{eq:summation}
	\mu_f(U,y)=\sum_{i=1}^N\mu_f(U_i,y)
\end{equation}
for pairwise disjoint domains $U_1,\ldots,U_N$ contained in $U$ and satisfying $$
y\notin f(\partial U_1)\cup\cdots\cup f(\partial U_N)\ \textrm{ and }
f\inv(y)\cap U\subset U_1\cup\cdots\cup U_N.
$$

As a consequence we obtain that, for a branched cover $f\colon X\to Y$, the local degree function $i_f:X\to\Z$, defined by 
\[
x\mapsto\mu_f(U,f(x)),
\]
where $U$ is any normal neighborhood of $x$, is well-defined. For branched covers, we may express the summation formula \eqref{eq:summation} in terms of the local index. Indeed, let $f:X\to Y$ be a branched cover between oriented cohomology manifolds of the same dimension and suppose $U\subset X$ is a normal domain for $f$. Then
\begin{equation}\label{summation}
	\mu_f(U,y)=\sum_{x\in f\inv(y)\cap U}i_f(x)
\end{equation}
for $y\in Y\setminus f(\partial U)$. If $f$ is a proper branched cover then
\begin{equation}
	\deg f=\sum_{x\in f\inv(y)}i_f(x)
\end{equation}
for any $y\in f(X)$; see \cite{ric93} and \cite{hei02}.

The local index satisfies a chain rule analogous to the chain rule for derivatives. More precisely, given branched covers $f:X\to Y$ and $g:Y\to Z$ between oriented cohomology manifolds, we have that 
\begin{equation}
	\label{eq:indexcomp} 
	i_{g\circ f}(x)=i_g(f(x))i_f(x) 
\end{equation}
for all $x\in X$.




A branched cover $f:X\to Y$ is \emph{sense preserving} (\emph{sense reversing}) if $\mu_f(D,y)>0$ ($\mu_f(D,y)<0$) for all precompact domains $D\subset X$ and $y\in f(D)\setminus f(\partial D)$. It is known that branched cover between oriented cohomology manifolds is either sense preserving or sense reversing \cite{vai66}. Thus we may always choose the orientations $c_X$ and $c_Y$ of $X$ and $Y$, respectively, so that a given branched cover $f$ is sense preserving. In particular we may assume $i_f\ge 1$ everywhere.

\subsubsection*{Branch set} Local homeomorphisms are always branched covers. However the converse fails, that is, a branched cover $f:X\to Y$ between oriented cohomology manifolds need not be a local homeomorphism. We define the set $B_f$ to be the set of points $x\in X$ for which $f$ is not a local homeomorphism at $x$. The branch set is easily seen to be a closed set. 

It is known that the branch set $B_f$ as well as its image $fB_f$ of a branched cover between oriented cohomology $n$-manifolds  has topological dimension at most $n-2$; see \cite{vai66}. In particular $B_f$ and $fB_f$ do not locally separate $X$ and $Y$, respectively, that is, $U\setminus B_f$ (resp. $V\setminus fB_f$) is path connected for every open $U\subset X$ (resp. $V\subset Y$); see also \cite[3.1]{hei02}.

An orientation preserving proper branched cover $f:X\to Y$ is $(\deg f)$-to-one in the sense that, for any $y\in Y\setminus f B_f$, the preimage $f\inv(y)$ contains exactly $\deg f$ points.

\subsection{BLD-maps and path-lifting}
A BLD-map $f:X\to Y$ between metric spaces $X$ and $Y$ is a branched cover satisfying the \emph{bounded length distortion} inequality (\ref{BLD}) for some $L\ge 1$. BLD-maps first appeared in \cite{mar88} as a subclass of quasiregular maps between Euclidean spaces, and in \cite{hei02} in the present metric context. We refer to \cite{lui17} for alternative characterizations of BLD-maps between metric spaces. 


A path-lifting yields a bijection between preimages of points not in the image of the branch set of the map. In what follows, we use the following version of \cite[Lemma 4.4]{lui17}. We omit the details.

\begin{lemma}\label{bijection}
	Let $f:X\to Y$ be an $L$-BLD map between two  oriented cohomology manifolds. Suppose there exists a geodesic joining $p,q\notin fB_f$. Let $K\subset X$ is a compact set. Then there is a bijection $\psi:f\inv(p)\to f\inv(q)$ satisfying
	\begin{equation}\label{bilip}
		d(p,q)/L\le d(x,\psi(x))\le Ld(p,q)
	\end{equation}
	for every $x\in f\inv(p)\cap K.$
\end{lemma}

\section{The pull-back of metric currents by BLD-maps}
\label{sec:pull-back}

\medskip\noindent Given a branched cover $f:X\to Y$ and set $E\subset X$ we say that a ball $B_r(y)\subset Y$ is a \emph{spread neighborhood} (of $y\in Y$) \emph{with respect to $E$} if $U_f(x,r)$ is a normal neighborhood of $x$ for each $x\in f\inv(y)\cap E$. We say that $B_r(y)$ is a \emph{spread neighborhood} if it is a spread neighborhood with respect to $X$.

Recall that, by Lemma \ref{spread}, for a compact set $K\subset X$, sufficiently small balls $B_r(y)$, for $y\in Y$ and $r>0$, are spread neighborhoods with respect to $K$. By Remark \ref{proper}, sufficiently small balls $B_r(y)$, for $y\in Y$ and $r>0$, are spread neighborhoods for proper BLD-maps. 

We say that a metric space $X$ is \emph{locally geodesic} if any point $x\in X$ has a neighborhood $U\subset X$ with the property that, for any two points $p,q\in U$, there is a geodesic joining them, i.e. a curve $\gamma:[0,d(p,q)]\to X$ satisfying $$d(p,q)=\ell(\gamma).$$ We call such neighborhoods \emph{geodesic neighborhoods}. Note, however, that the geodesic $\gamma$ is not required to lie inside the neighborhood $U$. We also say that a ball $B_r(y)\subset Y$ is a \emph{geodesic spread neighborhood with respect to a set $E\subset X$} if it is both a geodesic neighborhood, and a spread neighborhood with respect to $E$. Similarly, a \emph{geodesic spread neighborhood} is a spread neighborhood that is also a geodesic neighborhood.

We use the notation $\gamma:x\curvearrowright y$ to denote a curve $\gamma:[a,b]\to X$ joining two points $x,y\in X$.

\bigskip\noindent\emph{In what follows, we consider only locally geodesic oriented cohomology manifolds.}

\subsubsection*{Push-forward of functions by BLD-maps} Recall that the push-forward of a compactly supported Borel function $g:X\to\R$ by a BLD-map $f:X\to Y$ is the function $f_\sharp g:Y\to \R$, 
\[
y\mapsto \sum_{x\in f\inv(x)}i_f(x)g(x).
\]
It is not difficult to see that the push-forward $f_\sharp g$ is a Borel function.

\begin{lemma}\label{locpush}
	Let $f:X\to Y$ be a proper $L$-BLD map. 
	Given a Lipschitz function $\eta\in \LIP_c(X)$, the push-forward $f_\#\eta:Y\to \R$ is locally Lipschitz and satisfies the bound \[  \Lip f_\#\eta (y)\le L(\deg f)\Lip(\eta) \]for each $y\in Y$. Furthermore, $f_\#\eta$ satisfies the estimate 
	\[ 
	\|f_\#\eta\|_\infty\le (\deg f)\|\eta\|_\infty. 
	\]
\end{lemma}
\begin{proof}
	The  second estimate follows by a direct computation. Indeed, for any $p\in Y$, we have
	\begin{align*}
		|f_\#\eta(p)|\le \sum_{x\in f\inv(p)}i_f(x)\|\eta\|_\infty=(\deg f)\|\eta\|_\infty,
	\end{align*}
	by the summation formula (\ref{summation}) for the local index.

	We now prove the first estimate. Let $p\in Y$ and take a geodesic spread neighborhood $B_r(p)$ of $p$. The preimage $$f^{-1}B_r(p)=\bigcup_{x\in f^{-1}(p) }U_x$$ is a mutually disjoint union of normal neighborhoods $U_x$ of preimage points $x$. For any $q\in B_r(p)$ we have
	\[f_\#\eta(q)=\sum_{y\in U\cap f^{-1}(q)}i_f(y)\eta(y)=\sum_{x\in f^{-1}(p)}\sum_{y\in f^{-1}(q)\cap U_x }i_f(y)\eta(y)\] and further
	\begin{align}\label{star}
		|f_\#\eta(p)-f_\#\eta(q)|\le \sum_{x\in f^{-1}(p)}\left|i_f(x)\eta(x)-\sum_{y\in f^{-1}(q)\cap U_x }i_f(y)\eta(y)\right|.
	\end{align}
	By substituting the local summation formula (\ref{summation}) into (\ref{star}) we have the estimate
	\begin{align*}
		|f_\#\eta(p)-f_\#\eta(q)|\le &\sum_{x\in f^{-1}(p)}\left|\sum_{y\in f^{-1}(p)\cap U_x }i_f(y)\big(\eta(x)-\eta(y)\big)\right| \\
		\le & \sum_{x\in f^{-1}(p)}\sum_{y\in f^{-1}(p)\cap U_x }i_f(y)|\eta(x)-\eta(y)|\\
		\le & \Lip(\eta)\sum_{x\in f^{-1}(p)}\sum_{y\in f^{-1}(p)\cap U_x }i_f(y) d(x,y).
	\end{align*}
	For each $y\in U_x$, we have
	\[
	d(x,y)\le \ell(\gamma')\le L\ell(\gamma)=Ld(p,q),
	\]
	where $\gamma':x\curvearrowright y$ is a lift of a geodesic $\gamma:p\curvearrowright q$. Thus
	\begin{align*}
		|f_\#\eta(p)-f_\#\eta(q)|\le & \Lip(\eta) Ld(p,q)\sum_{x\in  f^{-1}(p)}\sum_{y\in f^{-1}(p)\cap U_x }i_f(y)\\
		=&Lr\Lip(\eta)\sum_{x\in f^{-1}(p)}i_f(x)=Ld(p,q)\Lip(\eta)\deg f.
	\end{align*}
	It follows that $\lip f_\#\eta(p)\le L\Lip(\eta)$ for every $p\in Y$. Suppose that $y\in Y$ and $B_r(y)$ is a spread neighborhood. Then, for any $p,q\in B_r(y)$, choosing a geodesic $\gamma$ connecting them, we have \[ |f_\#\eta(p)-f_\#\eta(q)|\le \int_0^1\Lip f_\#\eta(\gamma(t))|\dot\gamma_t|\ud t\le Ld(p,q)(\deg f)\Lip(\eta). \] This proves that $f_\#\eta$ is locally Lipschitz and satisfies the first estimate in the claim.
\end{proof}

The following lemma shows that the push-forward is natural  with respect to composition.
\begin{lemma}\label{compopush}
	Let $f:X\to Y$ and $g:Y\to Z$ be proper BLD-maps between locally geodesic oriented cohomology manifolds. Given a Borel function $h:X\to\R$ we have 
	\[ 
	g_\#(f_\#h)(z)=(g\circ f)_\#h(z) 
	\] 
	for every $z\in Z$.
\end{lemma}

\begin{proof}
	Let $z\in Z$. Then
	\[
	(g\circ f)\inv(z)=\bigcup_{y\in g\inv(z)}f\inv(y)
	\]
	and, by \eqref{eq:indexcomp}, $i_{g\circ f}(x)=i_g(f(x))i_f(x)$ for any $x\in (g\circ f)\inv(z)$.
	
	Thus
	\begin{align*}
		g_\# (f_\#h)(z)&=\sum_{y\in g\inv(z)}i_g(y)f_\#h(y)=\sum_{y\in g\inv(z)}\sum_{x\in f\inv(y)}i_g(y)i_f(x)h(x)\\
		&=\sum_{y\in g\inv(z)}\sum_{x\in f\inv(y)}i_g(f(x))i_f(x)h(x)=\sum_{x\in(g\circ f)\inv(z)}i_{g\circ f}(x)h(x)\\
		&=(g\circ f)_\#h(z)
	\end{align*}
	for every $z\in Z$.
\end{proof}

\subsubsection*{Push-forward of polylipschitz functions by BLD-maps} 
To simplify notation, we denote by $\overline x=(x_0,\ldots,x_k)$ a $(k+1)$-tuple of points in $X^{k+1}$. If $g:X\to \R$ is a function we define $\overline g:X^{k+1}\to\R$ by  $$\overline x\mapsto g(x_0)\cdots g(x_k).$$ For example, for the local index $i_f:X\to\Z$ of a BLD-map $f:X\to Y$, we denote $$\overline {i_f}(\overline x)=i_f(x_0)i_f(x_1)\cdots i_f(x_k)$$
for $\bar x=(x_0,\ldots,x_k)\in X^{k+1}$.

Let $f:X\to Y$ be a BLD-map between locally geodesic, oriented cohomology manifolds. Let $U\subset X$ be a normal domain for $f$. Given a normal domain $U\subset  X$ for $f$, consider the continuous $(k+1)$-linear linear map
\[
A_f^U:\LIP_\infty(U)^{k+1}\to \Poly^k(fU),\quad (\pi_0,\ldots,\pi_k)\mapsto ((f|_U)_{\#}\pi_0)\otimes \cdots \otimes ((f|_U)_{\#}\pi_k).
\]

\begin{definition}\label{polypush}
	Let $f:X\to Y$ be a BLD-map between locally geodesic, oriented cohomology manifolds, and let $U\subset X$ be a normal domain for $f$. The push-forward \[f_{U\#}:\Poly^k(U)\to\Poly^k(fU)\] is the unique continuous linear extension of $A_f^U$ for which (\ref{diagram:proj}) holds.
\end{definition}

By the linearity and the sequential continuity of $f_{U\#}$ we have that, if $\pi\in \Poly^k(U)$ and $(\pi_0^j,\ldots,\pi_k^j)\in\Rep(\pi)$, then
\begin{equation}\label{pushrep}
	f_{U\#}\pi=\sum_j^\infty((f|_U)_{\#}\pi_0^j)\otimes \cdots \otimes ((f|_U)_{\#}\pi_k^j).
\end{equation}

Denote $\deg(f|_U)=\mu_f(U)$ and let $L$ be the BLD-constant of $f$. Lemma \ref{locpush} and (\ref{pushrep}) immediately yield the estimates
\[
L_k(f_{U\#}\pi;fU)\le L^{k+1}\mu_f(U)^{k+1}\sum_j^\infty L(\pi^j_0|_{U})\cdots L(\pi^j_k|_{U}) 
\]
and
\[
\Lip_k(f_{U\#}\pi:fU)\le  L^{k}\mu_f(U)^{k+1}\sum_j^\infty\|\pi^j_0|_{U}\|_\infty\Lip(\pi^j_1|_{U}) \cdots \Lip(\pi^j_k|_{U})
\]
for every $\pi\in\Poly^k(U)$ and $(\pi_0^j,\ldots,\pi_k^j)\in\Rep(\pi)$. We obtain the following corollary.

\begin{corollary}\label{prest}
	Let $f:X\to Y$ an $L$-BLD-map between locally geodesic, oriented cohomology manifolds, and let $U\subset X$ be a normal domain for $f$. Then
	\begin{align*}
		&L_k(f_{U\#}\pi;fU)\le L^{k+1}\mu_f(U)^{k+1}L_k(\pi;U)
	\end{align*}
	and
	\begin{align*}
		\Lip_k(f_{U\#}\pi;fU)\le L^k\mu_f(U)^{k+1}\Lip_k(\pi;U).
	\end{align*}
\end{corollary}

\begin{lemma}\label{restr}(Restriction principle)
	Let $f:X\to Y$ be a BLD-map between locally geodesic oriented cohomology manifolds, and $x\in X$. Then, if $U'\subset U\subset X$ are normal domains for $x$ and if $\pi\in \Poly^k(U)$, we have \[ f_{U\#}\pi(\overline z)=f_{U'\#}\pi(\overline z) \] for all $\overline z\in f(U')^{k+1}$.
\end{lemma}
\begin{proof}
	By Lemma \ref{spread} (b), we have that $U\cap f\inv(f(U'))=U'$. Thus, for all $z\in f(U')$,
	\[
	U'\cap f\inv(z)=U\cap f\inv(z).
	\]
	The claim follows from this immediately.
\end{proof}

For the next three lemmas, we assume that $f:X\to Y$ is an $L$-BLD map between geodesic, oriented cohomology manifolds, $U\subset X$ is a normal domain for $f$, and that $k\ge 0$ is a fixed integer. We show that the push-forward commutes with the cup product and the exterior derivative. 

\begin{lemma}\label{precru}
	Given $\pi\in \Poly^k(U)$ and $\sigma\in \Poly^m(fU)$ we have 
	\[ 
	f_{U\#}(\pi\smile f^\#\sigma)=\mu_f(U)^m(f_{U\#}\pi)\smile \sigma. 
	\]
\end{lemma}

\begin{proof}
	We observe first that, given functions $g,h:U\to\R$ and $p\in fU$, we have
	\[
	(f|_U)_\#(g(h\circ f))(p)=\sum_{x\in U\cap f\inv (p)}i_f(x)g(x)h(f(x))=h(p)(f|_U)_\# g(p).
	\]
	Now let $\pi=\pi_0\otimes\cdots\otimes\pi_k\in \Poly^k(U)$ and $\sigma = \sigma_0\otimes\cdots\otimes\sigma_m\in \Poly^m(fU)$ be polylipschitz functions. Then
	\begin{align*}
		&f_{U\#}(\pi\smile f^\#\sigma)=f_{U\#}(\pi_0(\sigma_0\circ f)\otimes\pi_1\otimes\cdots\otimes\pi_k\otimes (\sigma_1\circ f)\otimes\cdots\otimes(\sigma_k\circ f))\\
		&=f_{U\#}(\pi_0(\sigma_0\circ f))\otimes f_{U\#}(\pi_1)\otimes \cdots\otimes f_{U\#}(\pi_k)f_{U\#}(\sigma_1\circ f)\otimes\cdots\otimes f_{U\#}(\sigma_k\circ f)\\
		&=(\sigma_0f_{U\#}\pi_0)\otimes f_{U\#}\pi_1\otimes\cdots\otimes f_{U\#}\pi_k\otimes (\mu_f(U)\sigma_1)\otimes\cdots\otimes(\mu_f(U)\sigma_m)\\
		&=\mu_f(U)^m(f_{U\#}\pi)\smile\sigma.
	\end{align*}
	Since the cup product is bi-linear and the pull-back is linear we have, by \eqref{pushrep}, that the claim holds  for all $\pi\in \Poly^k(U)$ and $\sigma\in \Poly^m(V)$.
\end{proof}

\begin{lemma}\label{prext}
	For each $\pi\in \Poly^k(U)$, we have 
	\[
	f_{U\#}(d\pi)=\mu_f(U)df_{U\#}\pi.
	\]
\end{lemma}
\begin{proof}
	As before, it suffices to consider the case $\pi=\pi_0\otimes\cdots\otimes\pi_k\in\Poly^k(U)$. Then
	\begin{align*}
		f_{U\#}(d\pi)
		=&\sum_{l=0}^{k+1} (-1)^lf_{U\#}(\pi_0\otimes\cdots\otimes \pi_{l-1}\otimes 1\otimes \pi_l\otimes\cdots\otimes\pi_k)\\
		=&\sum_{l=0}^{k+1}(-1)^lf_{U\#} \left( \pi_0\otimes\cdots\otimes \pi_{l-1}\right) \otimes \mu_f(U)\otimes f_{U\#}\left( \pi_l\otimes\cdots\otimes \pi_k\right)\\
		=&\mu_f(U)df_{U\#}\pi.
	\end{align*}
\end{proof}

The following lemma shows that the push-forward is sequentially continuous.

\begin{lemma}\label{precont}
	Suppose $\pi^n\to \pi$ in $\Poly^k(U)$. Then $f_{U\#}\pi^n\to f_{U\#}\pi$ in $\Poly^k(V)$.
\end{lemma}

\begin{proof}
	Since $\pi^n\to \pi$ in $\Poly^k(U)$ there is, for  every $n\in\N$, a representation
	\begin{align*}
		\pi^n-\pi=\sum_j^\infty\pi_0^{j,n}\otimes\cdots\otimes\pi_k^{j,n}
	\end{align*}
	of $\pi^n-\pi$ satisfying 
	\begin{align*}
		\sup_{n\in\N} \sum_j^\infty L(\pi_0^{j,n}|_U)\cdots L(\pi_k^{j,n}|_U)<\infty
		\ \text{ and }
		\lim_{n\to\infty}\sum_j^\infty\|\pi_0^{j,n}|_U\|_\infty\cdots\|\pi_k^{j,n}|_U\|_\infty= 0.
	\end{align*}
	Since
	\begin{align*}
		f_{U\#}\pi_n-f_{U\#}=f_{U\#}(\pi^n-\pi)=\sum_j(f_{U\#}\pi_0^{j,n})\otimes\cdots\otimes (f_{U\#}\pi_k^{j,n}),
	\end{align*}
	we have, by the estimates in Lemma \ref{locpush}, that
	\begin{align*}
		\sup_{n\in\N} & \sum_j^\infty L(f_{U\#}\pi_0^{j,n}|_V)\cdots L(f_{U\#}\pi_k^{j,n}|_V) \\
		\le& L^{k+1}\mu_f(U)^{k+1}\sup_{n\in\N}\sum_j^\infty L(\pi_0^{j,n}|_U)\cdots L(\pi_k^{j,n}|_U)<\infty
	\end{align*}
	and
	\begin{align*}
		\sum_j^\infty&\|f_{U\#}\pi_0^{j,n}|_V\|_\infty\cdots\|f_{U\#}\pi_k^{j,n}|_V\|_\infty 
		\le \mu_f(U)^{k+1}\sum_j^\infty \|\pi_0^{j,n}|_U\|_\infty\cdots\|\pi_k^{j,n}|_U\|_\infty\to 0
	\end{align*}
	as $n\to\infty.$ Thus $f_{U\#}\pi^n\to f_{U\#}\pi$ in $\Poly^k(V)$.
\end{proof}

Finally, we show that the push-forward is natural in the sense that the composition of push-forwards is the push-forward of compositions

\begin{lemma}\label{prenat}
	Let $f:X\to Y$ and $g:Y\to Z$ be BLD-maps between locally geodesic oriented cohomology manifolds. Let $U\subset X$ be a normal domain for $f$ and $V\subset f(U)$ a normal domain for $g$. Set $W=g(V)$ and $U'\subset f\inv(V)\cap U$ a component of $f\inv(V)\cap U$. Then \[ g_{V\#}\circ f_{U'\#}\pi=(g\circ f)_{U'\#}\pi \] for every $\pi\in \Poly^k(U')$.
\end{lemma}

\begin{proof}
	We observe first that $U'$ is a normal domain for $f$. Let $\pi=\pi_0\otimes\cdots\otimes\pi_k\in \Poly^k(U')$. By Lemma \ref{compopush}, we have 
	\[
	(g|_{V})_\#(f|_{U'})_\#\pi_j=\big((g\circ f)|_{U'}\big)_\#\pi_j  
	\] 
	on $W$, for each $j=0,\ldots,k$. Thus
	\begin{align*}
		(g_{V\#}\circ f_{U'\#})\pi=&g_{V\#}((f|_{U'})_\#\pi_0\otimes\cdots\otimes(f|_{U'})_\#\pi_k)\\
		=&((g|_{V})_\#(f|_{U'})_\#\pi_)\otimes\cdots\otimes((g|_{V})_\#(f|_{U'})_\#\pi_k)\\
		=&((g\circ f)|_{U'})_\#\pi_0\otimes\cdots\otimes((g\circ f)|_{U'})_\#\pi_k
		=(g\circ f)_{U'\#}\pi.
	\end{align*}
	By (\ref{pushrep}), equality holds for all $\pi\in \Poly^k(U')$.
\end{proof}

\subsection{Push-forward of polylipschitz forms}
\label{sec:push-forward_forms}

Let $f:X\to Y$ be a BLD-map between locally geodesic oriented cohomology manifolds $X$ and $Y$. We show that the push-forwards $f_{U\#}:\Poly^k(U)\to \Poly^k(f(U))$, where $U\subset X$ is a normal domain for $f$, induce a map $\Polys^k(X)\to \Polys^k(Y)$.
\begin{lemma}\label{wellposed}
	Let $f:X\to Y$ be a BLD-map between locally geodesic oriented cohomology manifolds, and let $x\in X$. Let  $U$ and $U'$ be normal neighborhoods of $x$, and let $\pi\in \Poly^k(U)$, $\pi'\in \Poly^k(U')$  be polylipschitz functions satisfying $[\pi]_x=[\pi']_x$. Then \[[f_{U\#}\pi]_{f(x)}= [f_{U'\#}\pi']_{f(x)}. \]
\end{lemma}
\begin{proof}
	We may assume $U'\subset U$. Since $[\pi]_x=[\pi']_x$, there exists $\rho>0$, for which $U(x,\rho)\subset U'$ and 
	\[
	\pi|_{U(x,\rho)^{k+1}}\equiv \pi'|_{U(x,\rho)^{k+1}}. 
	\]
	
	Since $U(x,\rho)$ is a normal neighborhood of $x$ we have, by the summation formula of the local index (\ref{summation}) that, for every $q\in B_\rho(p)$, \[U\cap f\inv(q)=U'\cap f\inv(q)=U(x,\rho)\cap f\inv(q).\]  Thus \[ f_{U\#}\pi|_{B_\rho(p)^{k+1}}= f_{U'\#}\pi'|_{B_\rho(p)^{k+1}},\]
	and
	\[
	[f_{U\#}\pi]_{f(x)}=[f_{U'\#}\pi']_{f(x)}.
	\]
	The claim follows.
\end{proof}
\begin{definition}\label{af}
	Let $f:X\to Y$ be a BLD-map between locally geodesic oriented cohomology manifolds. The \emph{local averaging map} $A_f:\Polys^k(X)\to \Polys^k(Y)$ is the map
	\[
	[\pi]_x\mapsto \frac{1}{i_f(x)^{k+1}}[f_{U_x\#}\pi]_{f(x)}
	\]
	where, for each $x\in X$, $U_x$ is a normal neighborhood of $x$.
\end{definition}
By Lemma \ref{wellposed}, the local averaging map $A_f:\Polys^k(X)\to \Polys^k(Y)$ is well-defined. Moreover, for each $x\in X$,
\[
A_f:\Polys_x^k(X)\to\Polys_{f(x)}^k(Y).
\]
\begin{remark}\label{linearity}
	For each $x\in X$, the stalks $\Polys^k_x(X)$ and $\Polys^k_{f(x)}(Y)$ are vector spaces. We have, by the linearity of $f_{U\#}$ that, for $[\pi]_x,[\pi']_x\in \Polys^k_x(X)$,
	\[
	A_f([\pi]_x+[\pi']_x)=A_f([\pi]_x)+A_f([\pi']_x).
	\]
\end{remark}
\begin{definition}\label{push}
	Let $\omega\in \G_c^k(X)$. The push-forward $f_\#\omega\in \G^k_c(Y)$ is the section 
	\[
	f_\#\omega:Y\to\Polys^k(Y),\quad y\mapsto \sum_{x\in f\inv(y)}i_f(x)A_f(\omega(x)). 
	\]
	
\end{definition}
Note that, since $\omega\in \G^k_c(X)$ has compact support, the sum in Definition \ref{push} has only finitely many nonzero summands.

Let $\omega\in\G^k_c(X)$ and $y\in Y$. The value of the push-forward $f_\#\omega$ at $y$ can be given as follows. Let $r>0$ be a radius with the property that  $B_r(y)$ is a geodesic spread neighborhood with respect to $\spt\omega$; cf. Lemma \ref{spread}. For each $x\in f\inv(y)\cap \spt\omega$, let $\pi_x\in \Poly^k(U(x,r))$ satisfy $[\pi_x]_x=\omega(x)$. Then 
\[ 
f_\#\omega(y)=[(A_f^r\omega)_y]_y, \textrm{ where }\
(A_f^r\omega)_y=\sum_{x\in f\inv(y)\cap \spt\omega}\frac{1}{i_f(x)^k}f_{U(x,r)\#}\pi_x\in \Poly^k(B_r(y)).
\]
Indeed, it suffices to note that
\[
\begin{split}
[(A_f^r\omega)_y]_y=&\sum_{x\in f\inv(y)\cap \spt\omega}\frac{1}{i_f(x)^k}[f_{U(x,r)\#}\pi_x]_y \\
=&\sum_{x\in f\inv(y)\cap \spt\omega} i_f(x)A_f([\pi_x]_x)=\sum_{x\in f\inv(y)}i_f(x)A_f(\omega(x))
=f_\#\omega(y).
\end{split}
\]
We use this fact in the sequel.


The next proposition lists the basic properties of the push-forward.

\begin{proposition}\label{basicpush} Let $f:X\to Y$ be an $L$-BLD map. The pushforward $\omega\mapsto f_\#\omega$ is a linear map $f_\#:\G_c^k(X)\to \G_c^k(Y)$ satisfying, for each $\omega\in \G_c^k(X)$, the following properties:
	\begin{itemize}
		\item[(1)] $\spt(f_\#\omega)\subset f(\spt\omega)$,
		\item[(2)] $\displaystyle L_k(f_\#\omega)\le L^{k+1}f_\# L_k(\omega)$ and $\displaystyle \|f_\#\omega\|\le L^kf_\#\|\omega\|$ pointwise on $Y$, 
		\item[(4)] $f_\#(d\omega)=df_\#\omega$, and
		\item[(3)] $f_\#(\alpha\smile f^\#\beta)=f_\#\alpha\smile\beta$ for $\alpha\in \G_c^k(X)$ and $\beta\in \G^m(Y)$.		
	\end{itemize}
\end{proposition}

\begin{proof}
	Linearity is straighforward to check (see Remark \ref{linearity}). Let $\omega\in\G_c^k(X)$ and $p\in Y$, $p\notin f(\spt\omega)$. Then $\spt\omega\cap f\inv(p)=\varnothing$ and therefore all the terms in the sum defining $f_\#\omega(p)$ are zero. This  proves (1).
	
	Let $B_r(p)$ be a geodesic spread neighborhood with respect to $\spt\omega$. By Corollary \ref{prest} we have
	\begin{align*}
		L_k(A^r_f\omega_p;B_r(p))\le &\sum_{x\in f\inv(p)}\frac{1}{i_f(x)^k}L^{k+1}i_f(x)^{k+1}L_k(\pi_x;U(x,r))\\
		= & L^{k+1}\sum_{x\in f\inv(p)}i_f(x)L_k(\pi_x;U(x,r)),
	\end{align*}
	where $\pi_x\in \Poly^k(U(x,r))$ satisfies $[\pi_x]_x=\omega(x)$ for $x\in \spt\omega\cap f\inv(p)$. Similarly 
	\begin{align*}
		\Lip_k(f^r_\#\omega_p;B_r(p))\le &\sum_{x\in f\inv(p)}\frac{1}{i_f(x)^k}L^{k}i_f(x)^{k+1}\Lip_k(\pi_x;U(x,r))\\
		= & L^{k}\sum_{x\in f\inv(p)}i_f(x)\Lip_k(\pi_x;U(x,r)).
	\end{align*}
	Taking the limit $r\to 0$ yields (2).
	
	To prove (3), we use Lemma \ref{prext}. We have
	\begin{align*}
		A^r_f(d\omega)_p&=\sum_{x\in f\inv(p)}\frac{1}{i_f(x)^{k+1}}f_{U(x,r)\#}(d\pi_x)=\sum_{x\in f\inv(p)}\frac{i_f(x)}{i_f(x)^{k+1}}df_{U(x,r)\#}\omega_x\\
		&=d\left(\sum_{x\in f\inv(p)}\frac{1}{i_f(x)^{k}}f_{U(x,r)\#}\pi_x\right)=dA^r_f\omega_p,
	\end{align*}
	for each $p\in Y$.
	
	For (4), let $p\in Y$ and let $\beta_p\in \Poly^k(B_r(p))$ be such that $[\beta_p]_p=\beta(p)$. For each $x\in \spt\omega\cap f\inv(p)$, choose polylipschitz functions $\alpha_x\in \Poly^k(U(x,r))$. By Lemma \ref{precru} we obtain
	\begin{align*}
		A^r_f(\alpha\smile f^\#\beta)_p&=\sum_{x\in f\inv(p)}\frac{1}{i_f(x)^{k+m}}f_{U(x,r)\#}(\alpha_x\smile f^\#\beta_p)\\
		&=\sum_{x\in f\inv(p)}\frac{i_f(x)^m}{i_f(x)^{k+m}}(f_{U(x,r)\#}\alpha_x)\smile \beta_p\\
		&=\left(\sum_{x\in f\inv(p)}\frac{1}{i_f(x)^{k}}f_{U(x,r)\#}\alpha_x\right)\smile \beta_p=(A_f^r\alpha_p)\smile \beta_p
	\end{align*}
	for each $p\in Y$. Thus
	\[
	f_\#(\alpha\smile f^\#\beta)(p)=[A_f^r(\alpha\smile f^\#\beta)_p]_p=[(A_f^r\alpha)_p\smile \beta_p]_p=\big((f_\#\alpha)\smile\beta\big) (p),
	\]
	for each $p\in Y$. 
\end{proof}

\subsection{Partition-continuity of the push-forward of polylipschitz forms}
\label{partcont} 
In general, $f_\#\omega$ need not be continuous for continuous $\omega\in\Gamma_c^k(X)$. However, $f_\#$ maps continuous sections to partition-continuous sections.
\begin{proposition}\label{pushphi}
	Let $f:X\to Y$ be a BLD-map between locally geodesic, oriented cohomology manifolds. Suppose $(\pi_0,\ldots,\pi_k)\in\sD^k(X)$, $\pi=\pi_0\otimes\cdots\otimes\pi_k\in \Poly_c^k(X)$ and $\omega=\imath(\pi)\in \Gamma_c^k(X)$. Then there is a finite Borel partition $\{ E_i \}_{j=1}^N$ of $Y$ for which 
	\[ 
	(f_\#\omega)|_{E_i}\in\Gamma^k(E_i) 
	\] 
	for each $j=1,\ldots,N$.
\end{proposition}

We prove Proposition \ref{pushphi} at the end of Section \ref{partcont}. For the proof, we briefly recall the monodromy representation of a proper branched covers.

Let $f:X\to  Y$ be a proper branched cover. Then there is a locally compact geodesic space $X_f$, a finite group $G=G_f$, called the \emph{monodromy group of $f$}, acting on $X_f$ by homeomorphisms, and a subgroup $H\le G$ satisfying 
\[ 
X_f/G\approx Y,\ U_f/H\approx X. 
\] 
The quotient maps 
\[ \overline f:X_f\to Y,\quad x\mapsto Gx,
\] 
and 
\[
\varphi:X_f\to X,\quad x\mapsto Hx,
\] 
are branched covers for which the diagram 
\begin{equation}\label{monodromy}
	\begin{tikzcd}
		X_f\arrow[swap]{d}{\varphi} \arrow[rd, "\bar f"]& \\
		X\arrow[swap]{r}{f} & Y
	\end{tikzcd}
\end{equation}
commutes. When $f$ is a BLD-map, the group $G$ acts on $X_f$ by bilipschitz maps, and $\overline f$ and $\varphi$ are BLD-maps. See \cite{aal17} and the references therein for details on monodromy representations.

The following multiplicity formula is a counterpart of \eqref{eq:indexcomp}. We refer to \cite{aal17} for similar multiplicity formulas.
\begin{lemma}\label{indexform}
	Let  $f:X\to Y$ be a proper branched cover. Consider the monodromy triangle (\ref{monodromy}) associated to $f$. For $w\in X_f$, denote by $H_w\le G_w$ the stabilizers of $w$ of $H$ and $G$, respectively. Then we have the identity \[ |G_w|=i_f(\varphi(w))|H_w| \]  for all $w\in X_f$. 
\end{lemma}
\begin{proof}
	Let $w\in X_f$ and let $W\subset X_f$ be a normal domain for $\bar f$. Then $\varphi(W)\subset X$ is a normal neighborhood of $\varphi(w)$ with respect to $f$. We denote $g=(\bar f)|_W:W\to \bar f(W)$. The stabilizers $G_w$ and $H_w$ act on $W$ and the restrictions $g$ and $\varphi|_W$ are orbit maps with respect to the action. Thus, the commuting diagram
	\[
	\begin{tikzcd}
	W\arrow[swap]{d}{\varphi|_W} \arrow[rd, "g"]& \\
	\varphi(W)\arrow[swap]{r}{f|_{\varphi(W)}} & \bar f(W)
	\end{tikzcd}
	\]
	is a monodromy representation of $f|_{\varphi(W)}$, with monodromy group $G_w$, and $\varphi|_W$ is the orbit map for $H_w$. Since $gB_g$ and $fB_f$ are nowhere dense, there exists $p'\in \bar f\setminus\big(gB_g\cup fB_f\big)$. Since $g$ is the orbit map for $G_w$, we have that
	\[
	|G_w|=|g\inv(p')|.
	\]
	On the other hand, $(f|_{\varphi(W)})\inv(p')\cap (\varphi|_W)B_{\varphi|_W}=\varnothing$. We conclude that
	\[
	\begin{split}
	|G_w|=|g\inv(p')|=|(\varphi|_W)\inv \big((f|_{\varphi(W)})\inv (p')\big)|=|H_w| \deg(f|_{\varphi(W)}).
	\end{split}
	\]
	Since $\deg(f|_{\varphi(W)})=i_f(w)$, the claim follows. 
\end{proof}

\subsubsection*{Fiber equivalence} Throughout this subsection we fix a proper BLD-map $f:X\to Y$. We introduce the fiber equivalence on $Y$ using the  monodromy representation (\ref{monodromy}) of $f$. Two points $p,q\in Y$ are said to be \emph{fiber equivalent}, $p\sim_f q$, if $|f\inv(p)|=|f\inv(q)|$ and there are labelings of the preimages 
\[ 
f\inv(p)=\{x_1,\ldots,x_m\}\quad\textrm{and } f\inv (q)=\{y_1,\ldots,y_m \} 
\] 
satisfying 
\[ 
|\varphi\inv(x_j)|=|\varphi\inv(y_j)|\quad \textrm{for each } j=1,\ldots,m. 
\]

\begin{lemma}\label{claim3}
	The equivalence relation $\sim_f$ has finitely many equivalence classes, each of which is a Borel set.
\end{lemma}

\begin{proof}
	The equivalence classes of $\sim_f$ are
	\[
	\{E_m(k_1,\ldots,k_m): m\in\N,\ k_1,\ldots,k_m\in \N \},
	\]
	where
	\begin{align*}
		E_m (k_1,\ldots,k_m) =\{p\in Y:f|_U\inv(p)=\{x_1,\ldots,x_m\},\ |\varphi\inv(x_j)|=k_j,\ j=1,\ldots,m \}
	\end{align*}
	for each $m\in \N$ and $k_1,\ldots,k_m\in\N$. Since the sets $E_m(k_1,\ldots,k_m)$ are empty if $k_j>|G|$ or $m>\deg f$, we find that there are only finitely many equivalence classes of $\sim_f$.
	
	To see that each of the sets $E_m(k_1,\ldots,k_m)$ is Borel, set
	\[ 
	E_m=\{ p\in Y: |f\inv(p)|=m \} 
	\] 
	and 
	\[ 
	E(k)=f(A(k)),\ A(k)=\{ x\in X: |\varphi\inv(x)|=k \}. 
	\] 
	The sets $E_m$ and $E(k)$ are clearly Borel. Observe that 
	\[ 
	E_m(k_1,\ldots,k_m)=E_m\cap E(k_1)\cap \cdots\cap E(k_m), 
	\] 
	whence the Borel measurability of the equivalence classes follows.
\end{proof}

\begin{lemma}\label{claim1}
	If $p\sim_f q$, then $|\bar f\inv(p)|=|\bar f\inv(q)|$ and $i_f(x_j)=i_f(y_j)$ for each $j=1,\ldots,m$.
\end{lemma}

\begin{proof}
	Let $z\in \bar f\inv(p)$ and $z'\in \bar f\inv(q)$. Then \[ |\bar f\inv(p)|=|Gz|=\sum_j^m|\varphi\inv(x_j)|=\sum_j^m|\varphi\inv(y_j)|=|Gz'|=|\bar f\inv(q)|.  \] For  second claim let $H_w\le G_w$ be the stabilizer subgroups of $w\in X_f$ in $H$ and $G$, respectively. Since $\varphi$ is an orbit map of the action $H\curvearrowright X_f$, we have that $|Hw|=|\varphi\inv\varphi(w)|.$ Hence
	\begin{equation*}\label{f1}
		|H_w|=\frac{|H|}{|Hw|}=\frac{|H|}{|\varphi\inv(\varphi(w))|}.
	\end{equation*}
	Thus, by Lemma \ref{indexform}, we have, for $w=z_j\in \varphi\inv (x_j)$ and $w=z'_j=\varphi\inv(y_j)$, that 
	\[ 
	i_f(x_j)=\frac{|G_{z_j}|}{|H_{z_j}|}=\frac{|G|}{|G_{z_j}|}\frac{|\varphi\inv(x_j)|}{|H|}=\frac{|G|}{|H|}\frac{|\varphi\inv(x_j)|}{|\bar f\inv(p)|} =\frac{|G|}{|H|}\frac{|\varphi\inv(y_j)|}{|\bar f\inv(q)|}=i_f(y_j)
	\] 
	for each $j=1,\ldots,m$.
\end{proof}

\begin{lemma}\label{claim2}
	Let $p\sim_f q$. Let $B_r(p)$ and $B_s(q)$ be spread neighborhoods for $f$ and $\bar f$. Then 
	\[ 
	f\inv(B_r(p)\cap B_s(q))=\bigcup_{j=1}^m \left( U_f(x_j,r)\cap U_f(y_j,s)\right). 
	\]
\end{lemma}

\begin{proof}
	For each $j=1,\ldots,m$ let $k_j=|\varphi\inv(x_j)|=|\varphi\inv(y_j)|$ and let
	\[ 
	\varphi\inv(x_j)=\{z_j^1,\ldots,z_j^{k_j}\},\ \varphi\inv(y_j)=\{w_j^1,\ldots,w_j^{k_j}\}. 
	\] 
	Since $|G_{z_j^l}|=|G_{w_j^l}|$ for all $j=1,\ldots,m$ and $l=1,\ldots,k_j$, we have that 
	\[ 
	\bar f\inv(B_r(p)\cap B_s(q))=\bigcup_{j=1}^m\bigcup_{l=1}^{k_j}\big(U_{\bar f}(z_j^l,r)\cap U_{\bar f}(w_j^l,s)\big). 
	\] 
	The claim now follows from the identity $f\inv A=\varphi(\bar f\inv A)$ for $A\subset Y$. 
\end{proof}

\begin{proof}[Proof of Proposition \ref{pushphi}]
	Suppose first that $\spt\pi_0\subset U$ where $U\subset X$ is a normal domain for $f$. Then $f_\#(\omega)(p)=0$ whenever $p\in Y\setminus f(U)$.
	
	Define the partition $\mathcal E$ on $f(U)$ as the collection of equivalence classes of the fiber equivalence relation $\sim\ :=\ \sim_{f|_U}$ related to $f|_U$.

	Let $p,q\in f(U)$ be fiber equivalent, that is $p\sim q$, and let $B_r(p)$ and $B_r(q)$ be geodesic spread neighborhoods for $f|_U$ and $\overline{f|_U}$. Then, by Lemma \ref{claim2}, \[ U_f(x_j,r)\cap (f|_U)\inv(z)=U_f(y_j,s)\cap (f|_U)\inv(z) \] for $j=1,\ldots,m$ and $z\in B_r(p)\cap B_s(q)$.
	
	It follows that, for each $\overline z=(z_0,\ldots,z_k)\in (B_r(p)\cap B_s(q))^{k+1}$, we have 
	\[ 
	(f|_{U(x_j,r)})_\#\pi_l(z_l)=(f|_{U(y_j,s)})_\#\pi_l(z_l) 
	\] 
	for all $l=0,\ldots, k$. Thus, by Lemma \ref{claim1},
	\begin{align*}
		A^r_f\omega_p(\overline z)=&\sum_{j=1}^m\frac{1}{i_f(x_j)^k}(f|_{U(x_j,r)})_\#\pi_0(z_0)(f|_{U(x_j,r)})_\#\pi_1(z_1)\cdots (f|_{U(x_j,r)})_\#\pi_k(z_k)\\
		=&\sum_{j=1}^m\frac{1}{i_f(y_j)^k}(f|_{U(y_j,s)})_\#\pi_0(z_0)(f|_{U(y_j,s)})_\#\pi_1(z_1)\cdots (f|_{U(y_j,s)})_\#\pi_k(z_k)\\
		=&A^s_f\omega_q(\overline z)
	\end{align*}
	for all $\bar z\in(B_r(p)\cap B_s(q))^{k+1}$.
	
	For every $p\in U$ choose a radius $r_p>0$ such that $B_{r_p}(p)$ is a geodesic spread neighborhood for $f$ with respect to $\spt\omega$. We have proved that $\{A^{r_p}_f\omega_p \}_{B_{r_p}(p)}$ satisfies the overlap condition (1) in \cite[Definition 6.2]{teripekka1} for every equivalence class of the fiber equivalence. Condition (2) in \cite[Definition 6.2]{teripekka1} follows from the first estimate in Corollary \ref{prest}. Indeed, for each $p\in Y$, we have 
	\begin{align*}
		&L_k(A^r_f\omega_p;E\cap B_r(p))\le \sum_{x\in f\inv(p)}\frac{L^{k+1}i_f(x)^{k+1}}{i_f(x)^k}L_k(\pi_0\otimes\cdots\otimes\pi_k;E\cap U(x,r))\\
		&\le L^{k+1} \left(\sum_{x\in \spt\omega\cap f\inv(p)}i_f(x)\right) L(\pi_0|_{B(\spt\pi_0,r)})\cdots L(\pi_k|_{B(\spt\pi_0,r)}).
	\end{align*}
	By Lemma \ref{claim3}, $\mathcal E$ is a finite Borel partition. Thus $f_\#\omega\in \Parcon(X)$.

	We have demonstrated that $f_\#\omega$ is $\mathcal E$-continuous under the assumption that  $\spt\pi_0$ is contained in a normal domain for $f$. Suppose $(\pi_0,\ldots,\pi_k)\in\sD^k(X)$ and let  $\omega=\imath(\pi_0\otimes\cdots\otimes\pi_k)\in \Gamma_c^k(X)$. Set $K=\spt\omega$. Let $\mathcal U=\{ U_1,\ldots, U_M \}$ be a finite covering of $K$ by normal domain for $f$ and let  $\{ \varphi_1,\ldots,\varphi_{M+1} \}$ be a Lipschitz partition of unity subordinate to $\mathcal U\cup \{X\setminus K\}$ satisfying  $\spt\varphi_{M+1}\subset X\setminus K$.
	
	For each $l=1,\ldots,M+1$,  $\spt(\varphi_l\pi_0)$ is contained in a normal domain for $f$. Thus $f_\#(\varphi_l\omega)$ is $\mathcal E$-continuous. We conclude that the finite sum 
	\[ 
	f_\#\omega=\sum_{l=1}^Mf_\#(\varphi_l\omega) 
	\] 
	is $\mathcal E$-continuous. 
\end{proof}

\subsection{Pull-back of currents of locally finite mass by BLD-maps} To define the pull-back of a $k$-current $T\in M_{k,loc}(X)$ as $T\circ f_\#$ (see the discussion in the introduction) it remains to show that the resulting functional is weakly continuous. 

\begin{proposition}\label{pushcont}
	Let $f:X\to Y$ be a BLD-map between locally geodesic, oriented cohomology manifolds. If $\pi^n\to \pi$ in $\Poly_c^k(X)$ then $f_\#\pi^n\to f_\#\pi$ in $\Parcon(X)$.
\end{proposition}
\begin{proof}
	Let $K\subset X$ be a compact set containing $\spt\pi^n_0$ for each $n\in\N$. Let $\mathcal U=\{U_1,\ldots,U_M\}$ be an open cover of $K$ by normal neighborhoods, and let $\{\varphi_1,\ldots,\varphi_{M+1} \}$ be a Lipschitz partition of unity subordinate to $\{X\setminus K\}\cup \mathcal U$. It suffices to prove \[ f_\#(\varphi_l\smile\pi^n)\to f_\#(\varphi_l\smile\pi) \] in $\Parcon(X)$ for each $l=0,\ldots, M$.
	
	Fix $U=U_l$ and let $\mathcal E=\{E_1,\ldots,E_N\}$ be the equivalence classes of the fiber equivalence $\sim_{f|_U}$ associated to $f|_U$. Let $\{B_1,\ldots,B_Q\}$ be a finite cover of $\overline{fU}$ such that each $B_i=B_{r_i}(p_i)$ is a geodesic spread neighborhood of $f|_U$. For any $E=E_m(\mu_1,\ldots,\mu_m)\in \mathcal E$, we may write $f|_U\inv (p_i)=\{x_1,\ldots,x_m\}$, where $i_f(x_l)=\mu_l$; see the proof of Lemma \ref{claim3}. Set 
	\[
	\sigma^n=\varphi_U\smile(\pi^n-\pi).
	\]
	We have
	\begin{align*}
		A^{r_i}_f(\sigma^n) =\sum_{l=1}^m\frac{1}{i_f(x_l)^k}f_{U(x_l,r_i)\#}\sigma^n.
	\end{align*}
	Since $\sigma^n\to 0$ in $\Poly_U^k(X)$, and hence the restrictions converge in $\Poly^k(U(x_l,r_i))$, it follows from Lemma \ref{precont} that $\displaystyle f^i_\#(\sigma^n)|_{(E\cap B_i)^{k+1}}\to 0$ in $\Poly^k(E\cap B_i)$. Furthermore,
	\begin{align*}
		L_k(f^i_\#\sigma^n;E\cap B_i)\le \sum_{l=1}^m\frac{L^{k+1}\mu_f(U(x_l,r_i))}{\mu_l^k}L_k(\sigma^n;B_i)\le L^{k+1}\mu_f(U) \sup_{n\in\N} L_k(\sigma^n;U)
	\end{align*}
	for all $i\in\N$ and $E\in\mathcal E$. This shows that $f_\#(\varphi_U\smile\pi^n)\to f_\#(\varphi_U\smile\pi)$ in $\Parcon(X)$. The claim follows.
\end{proof}

We now define the pull-back of currents of locally finite mass. 

\begin{definition}\label{pull}
	Let $f:X\to Y$ be a BLD-map between two locally geodesic, oriented cohomology manifolds $X$ and $Y$, and let $T\in M_{k,\loc}(Y)$ be a $k$-current of locally finite mass on $Y$. The \emph{pullback} $f^\ast T\in \sD^k(X)$ of $T$ is the $k$-current \[ (\pi_0,\ldots,\pi_k)\mapsto \widehat T(f_\#(\pi_0\otimes\cdots\otimes\pi_k)) \] for every $(\pi_0,\ldots,\pi_k)\in \sD^k(X)$.
\end{definition}

\begin{proposition}\label{massest}
	Let $f:X\to Y$ be an $L$-BLD-map between locally geodesic, oriented cohomology manifolds, and let $T\in M_{k,\loc}(Y)$. Then $f^\ast T\in M_{k,\loc}(X)$ and 
	\[
	\|f^\ast T\|\le L^kf^\ast\|T\|. 
	\] 
\end{proposition}
\begin{proof}
	By Lemma \ref{pushcont} and Theorem \ref{ext}, $f^\ast T$ is sequentially continuous. Let $(\pi_0,\ldots,\pi_k)\in \sD^k(X)$ and $\pi=\pi_0\otimes\cdots\otimes\pi_k$. Suppose $E_1,\ldots,E_N$ is a Borel partition of $Y$ for which $$\displaystyle f_\#\pi|_{E_i}\in\Gamma^k(E_i)$$ for each $i=1,\ldots,N$. By \cite[Proposition 6.7]{teripekka1} and  Lemma \ref{basicpush}, we may estimate
	\begin{align*}
		|f^\ast T(\pi)|=&|\widehat T(f_\#\pi)|\le \sum_{i=1}^N\int_{E_i}\|f_\#\pi|_{E_i}\|_p\ud\|T\|(p)
		\le \sum_{i=1}^N\int_{E_i}L^kf_\#\|\pi\|(p)\ud\|T\|(p) \\
		= &L^k\int_Yf_\#\|\pi\|\ud\|T\|
		=L^k\int_X\|\pi\|\ud(f^\ast\|T\|).
	\end{align*}
	This proves $f^\ast T$ is a $k$-current of locally finite mass and provides the desired estimate.
\end{proof}

\begin{proof}[Proof of Theorem \ref{thm: locpull}]
	Properties (1) and (2) follow directly from the corresponding properties (3) and (4) in Proposition \ref{basicpush}. Indeed, let $T\in M_{k,\loc}(Y)$. Then
	\[
	\begin{split}
	f^\ast(\partial T)(\pi)&=\widehat{\partial T}(f_\#\pi)=\partial\widehat T(f_\#\pi)=\widehat T(df_\#\pi)=\widehat T(f_\#(d\pi))\\
	&=\partial\big(\widetilde T\circ f_\#\big)(\pi)=\partial(f^\ast T)(\pi)
	\end{split}
	\]
	and 
	\[ 
	f_\ast((f^\ast T)|_E)(\sigma)=f^\ast T(\chi_E\smile f^\#\sigma)=\widehat T(f_\#(\chi_E\smile f^\#\sigma))=\widehat T(f_\#\chi_E\smile \sigma), 
	\] 
	for all $\pi\in \sD^k(X)$, $\sigma\in \sD^k(Y)$ and Borel set $E\subset X$.
	
	To establish (3), let $E\subset X$ be a Borel set. We use (2) together with \cite[Lemma 4.6]{lan11} to conclude that \[ \|T\|\lfloor_{f_\#\chi_E}=\|T\lfloor_{f_\#\chi_E}\|=\|f_\ast((f^\ast T)\lfloor_E)\|\le L^kf_\ast(\|f^\ast T\|\lfloor_E). \] Thus
	\begin{align*}
		\int_E\ud(f^\ast\|T\|)=\int_Yf_\#\chi_E\ud\|T\|\le L^k\int_X\chi_E\ud\|f^\ast T\|=L^k\int_E\ud \|f^\ast T\|.
	\end{align*}
	Hence $f^\ast\|T\|(E)\le L^k \|f^\ast T\|(E)$ for all Borel sets $E\subset X$. We conclude that 
	\[
	f^\ast\|T\|\le L^k \|f^\ast T\|. 
	\] 
	The converse inequality is proven in Proposition \ref{massest}. Naturality is proven in the next subsection, Proposition \ref{naturality}.
\end{proof}

\subsection{Uniqueness and naturality of the pull-back} We now turn our attention to the uniqueness and naturality of the pull-back. This section also contains the proof of Theorem \ref{unique}.

\begin{proposition}\label{preunique}
	Let $f:X\to Y$ be a BLD-map between locally geodesic, oriented cohomology manifolds. Let $T\in M_{k,loc}(X)$ be a $k$-current for which $$f_\ast(T\lfloor_E)=0$$ for every precompact Borel set $E\subset X$. Then $T=0$. 
\end{proposition}

We begin with an auxiliary lemma on BLD-maps.
\begin{lemma}\label{restbilip}
	Let $f:X\to Y$ be a $L$-BLD-map between locally geodesic, oriented cohomology manifolds, and let $U\subset X$ be a normal domain for $f$. For each $m\in\N$, let $$E_m=U\cap i_f\inv(m).$$ Then $f|_{E_m}$ is locally $L$-bilipschitz.
\end{lemma}

\begin{proof}
	Let $x\in E_m$ and $r>0$ be such that $U(x,2Lr)\subset U$ is a normal domain for $x$. We claim that $f|_{E_m\cap U(x,r)}$ is $L$-bilipschitz.
	
	First we show that $f|_{E_m\cap U(x,2Lr)}$ is injective. Let $y\in E_m\cap U(x,2Lr).$ Then $m=i_f(y)=i_f(x)$. Suppose $f\inv(f(y))\cap U(x,2Lr)\ne \{ y\}$. Then, by  the summation formula (\ref{summation}), we have
	\begin{align*}
		i_f(x)=\mu_f(U(x,2rL),x)=\mu_f(y,U(x,2rL))=\sum_{x'\in f\inv(f(y))\cap U(x,2Lr)}i_f(x')>i_f(y),
	\end{align*}
	since $i_f(x')\ge 1$ for each $x'\in X$. This is a contradiction, whence $$ f\inv(f(y))\cap U(x,2Lr)=\{y\} $$ and $f|_{E_m\cap U(x,2Lr)}$ is injective.
	
	The fact that $f|_{E_m\cap U(x,r)}$ is $L$-Lipschitz is clear. Moreover the proof of injectivity shows that
	\[
	E_m\cap U(x,r)=f\inv(f(E_m\cap U(x,r))).
	\]
	Let $z',w'\in f(E_m\cap U(x,r))$, and let $z,w\in E_m\cap U(x,r)$ satisfy $z'=f(z)$ and $w'=f(w)$. Suppose $\gamma$ is a geodesic joining $z'$ and $w'$ in $B_{2r}(f(x))$. Since $f\inv(w')\cap U(x,r)=\{w\}$, we have that a lift $\gamma'$ in $U(x,2Lr)$ of $\gamma$ starting at $z$ ends at $w$. Thus \[ d(z,w)\le \ell(\gamma')\le L\ell(f\circ\gamma)=L\ell(\gamma)=Ld(z',w'). \] This finishes the proof of the claim.
\end{proof}

\begin{proof}[Proof of Proposition \ref{preunique}]
	Let $x\in X$. If $x\notin B_f$, let $r>0$ be such that $f|_{B_r(x)}$ is bilipschitz. Then 
	\[
	T\lfloor_{B_r(x)}=(f|_{B_r(x)}\inv)_\ast (f|_{B_r(x)})_\ast(T\lfloor_{B_r(x)})=(f|_{B_r(x)}\inv)_\ast f_\ast(T\lfloor_{B_r(x)})=0.
	\]
	
	Suppose now that $x\in B_f$. Let $r>0$ be a radius for which $U(x,Lr)$ is a normal neighborhood of $x$ and for which $f|_{U(x,r)\cap i_f\inv(m)}$ is $L$-bilipschitz. Then 
	\[
	B_r(x)=\bigcup_{m=1}^{i_f(x)}(E_m\cap B_r(x)),
	\]
	where $E_m=U(x,Lr)\cap i_f\inv(m)$. Let $m\le i_f(x)$. By the same argument as above 
	\[ 
	T\lfloor_{E_m\cap B_r(x)}=(f|_{E_m\cap B_r(x)})\inv f_\ast(T\lfloor_{E_m\cap B_r(x)})=0. 
	\]  
	Thus 
	\[ 
	T\lfloor_{B_r(x)}=\sum_{m=1}^{i_f(x)}T\lfloor_{E_m\cap B_r(x)}=0. 
	\] 
	We have proven that, for each $x\in X$, there exists a radius $r>0$ such that 
	\[
	T\lfloor_{B_r(x)}=0.
	\]
	Thus $T=0$.
\end{proof}

\begin{proof}[Proof of Theorem \ref{unique}]
	By Theorem \ref{thm: locpull} (1) and the assumption on $S$, we have \[ f_\ast[(f^\ast T-S)\lfloor E]=0 \] for every precompact Borel set $E\subset X$. Proposition \ref{preunique} implies that \[ f^\ast T-S=0, \] completing the proof.
\end{proof}

We now prove the naturality of the pull-back. The first auxiliary result is the naturality of the push-forward of polylipschitz forms.

\begin{proposition}\label{natpush}
	Let $f:X\to Y$ and $g:Y\to Z$ be BLD-maps between locally geodesic, oriented cohomology manifolds $X,Y$ and $Z$. Let $(\pi_0,\ldots,\pi_k)\in\sD^k(X)$ and  $\pi=\pi_0\otimes\cdots\otimes\pi_k\in \Poly_c^k(X)$. Then \[ g_\#(f_\#\pi)=(g\circ f)_\#\pi \] as sections in $\G_c^k(Z)$.
\end{proposition}
\begin{proof}
	The composition $g\circ f:X\to Z$ is a BLD-map. Denote by $L$ be the maximum of the BLD constans of $f$ and $g$. Let $(\pi_0,\ldots,\pi_k)\in\sD^k(X)$ and  $\pi=\pi_0\otimes\cdots\otimes\pi_k\in \Poly_c^k(X)$. Denote $\omega=f_\#\pi\in \G_c^k(Y)$, and $\spt\pi=K$. Let $q\in Z$ and fix $r>0$ for which $B_r(q)$ is a geodesic spread neighborhood for $g\circ f$ with respect to $K$, and $U_{g}(y,r)$ is a geodesic spread neighborhood of $y$, for each $y\in g\inv(q)\cap f(K)$, with respect to $K$.
	Since  \[(g\circ f)\inv(B_r(q))=\bigcup_{y\in g\inv(q)}f\inv(U_g(y,r))=\bigcup_{x\in (g\circ f)\inv(q)}U_{g\circ f}(x,r), \] we have that, for each $y\in g\inv(q)\cap f(K)$, the set  $f\inv(U_g(y,r))$ is a pairwise disjoint union \[ f\inv(U_g(y,r))=\bigcup_{x\in f\inv(y)}U_{g\circ f}(x,r). \]
	Thus, each $U_{g\circ f}(x,r)$ is a normal neighborhood of $x$ with respect to $f$. For each $y\in g\inv(q)\cap f(K)$,  let
	\[ 
	\sigma_y=(A_f^r\pi)_y=\sum_{x\in f\inv(y)\cap K}\frac{1}{i_f(x)^k}f_{U_{g\circ f}(x,r)\#}\pi; 
	\] 
	see the discussion after Definition \ref{push}.  Then $g_\#(f_\#\pi)(q)=[A^r_g\omega_q]_q$, where 
	\[ 
	A^r_g\omega_q=\sum_{y\in g\inv(q)\cap f(K)}\frac{1}{i_g(y)^k}g_{U_g(y,r)\#}\sigma_y.
	\]

	For each $y\in f\inv(q)\cap f(K)$ we have, by Lemma \ref{prenat}, that
	\begin{align*}
		g_{U_g(y,r)\#}\omega_y&=\sum_{x\in f\inv(y)\cap K}\frac{1}{i_f(x)^k}g_{U_g(y,r)\#}f_{U_{g\circ f}(x,r)\#}\pi\\
		&= \sum_{x\in f\inv(y)\cap K}\frac{1}{i_f(x)^k}(g\circ f)_{U_{g\circ f}(x,r)\#}\pi.
	\end{align*}
	Thus we have
	\begin{align*}
		A^r_g\omega_q&=\sum_{y\in g\inv(q)\cap f(K)}\sum_{x\in f\inv(y)\cap K}\frac{1}{i_g(y)^ki_f(x)^k}(g\circ f)_{U_{g\circ f}(x,r)\#}\pi\\
		&=\sum_{y\in g\inv(q)\cap f(K)}\sum_{x\in f\inv(y)\cap K}\frac{1}{i_{g\circ f}(x)^k}(g\circ f)_{U_{g\circ f}(x,r)\#}\pi\\
		&=\sum_{x\in g\circ f\inv(q)\cap K}\frac{1}{i_{g\circ f}(x)^k}(g\circ f)_{U_{g\circ f}(x,r)\#}\pi	
		=A_{g\circ f}^r\pi
	\end{align*}
	on $B_r(q)^{k+1}$. From this we conclude that \[ g_\#(f_\#\pi)(q)=(g\circ f)_\#\pi(q) \] for all $q\in Z$. 
\end{proof}

Proposition \ref{natpush} yields the naturality of the push-forward $f_\#:\G_c^k(X)\to \G_c^k(Y)$ for a BLD map $f:X\to Y$. We use this and Theorem \ref{unique} to conclude the naturality of the pull-back of metric currents.

\begin{proposition}\label{naturality}
	Let $f:X\to Y$ and $g:Y\to Z$ be BLD-maps between locally geodesic, oriented cohomology manifolds. Then \[ (g\circ f)^\ast= f^\ast\circ g^\ast \] as maps $M_{k,loc}(Z)\to M_{k,loc}(X)$.
\end{proposition}

Note that, if $T\in M_{k,loc}(Z)$, then, by definition and Proposition \ref{natpush}, we have \[ (g\circ f)^\ast T(\pi)=T((g\circ f)_\#\pi)=T(g_\#f_\#\pi) \] for $\pi\in \sD^k(X)$. Unfortunately, since $f_\#\pi\in \G^k_c(Y)$ is not necessarily in $\Gamma^k_c(Y)$, we cannot conclude that $g^\ast T(f_\#\pi)$ (strictly speaking, $\widehat{g^\ast T}(f_\#\pi)$) is given by $T(g_\#f_\#\pi)$.

\begin{proof} We use Theorem \ref{unique}. Let $\eta\in \LIP_c(X)$ and $\pi\in \sD^k(Z)$. Then, by \cite[Lemma 4.14 (a)]{teripekka1}, Proposition \ref{basicpush} (4) and Proposition \ref{natpush}, 
	\begin{align*}
		(g\circ f)_\ast (f^\ast g^\ast T\lfloor \eta)(\pi)&=(f^\ast g^\ast T)(\eta\smile (g\circ f)^\#\pi)=T(g_\#f_\#(\eta\smile f^\# g^\#\pi))\\
		&=T(g_\#(f_\#\eta\smile g^\#\pi))=T(g_\#f_\#\eta\smile \pi)\\
		&=\big(T\lfloor {(g\circ f)_\#\eta}\big)(\pi).
	\end{align*}
	Let $E\subset X$ be a Borel set. Let $\eta_j$ be a sequence in $\LIP_c(X)$ converging to $\chi_E$ in $L^1(\|T\|)$ and $\|T\|$-almost everywhere; see \cite[Proposition 2.3.13 and Remark 2.3.16(a)]{HKST07}.  Then, for each $\pi\in \sD^k(Z)$,
	\[\begin{split}
	&(g\circ f)_\ast (f^\ast g^\ast T\lfloor E)(\pi)=\lim_{n\to\infty}(g\circ f)_\ast (f^\ast g^\ast T\lfloor \eta_n)(\pi)\\
	&=\lim_{n\to\infty}\big(T\lfloor {(g\circ f)_\#\eta_n}\big)(\pi)=T\lfloor(g\circ f)_\#\chi_E(\pi).
	\end{split}\]
	Thus, by Theorem \ref{unique}, we have that \[ f^\ast g^\ast T=(g\circ f)^\ast T. \] This completes the proof.
\end{proof}

\subsection{Pull-back of proper BLD maps} Throughout this subsection, $f:X\to Y$ is a proper $L$-BLD map between geodesic, oriented cohomology manifolds $X$ and $Y$. Recall that a proper branched cover is $(\deg f)$-to-one; see the discussion on the branch set in Section 6.3.

In this subsection we prove Corollary \ref{thm:pull-back}, that is, we prove that the pull-back $f^*:M_{k,\loc}(Y)\to M_{k,\loc}(X)$ satisfies the following properties.
\begin{enumerate}
	\item the composition $f_* \circ f^* \colon M_{k,\loc}(Y) \to M_{k,\loc}(Y)$ satisfies $f_\ast \circ f^\ast = (\deg f) \id$; 
	\item the pull-back $f^\ast$ commutes with the boundary, i.e. $\partial f^\ast T=f^\ast(\partial T)$ for $T\in M_{k,\loc}(X)$; and
	\item for each $T\in M_k(X)$, 
	\[
	\frac{1}{L^{k}}(\deg f)M(T)\le M(f^\ast T)\le L^k(\deg f)M(T).
	\] 
\end{enumerate}
Moreover,  the pull-back operator $f^\ast$ restricts to an operator 
\[ 
f^\ast: N_{k}(Y)\to N_{k}(X). 
\]

\begin{proof}[Proof of Corollary \ref{thm:pull-back}]
	Besides the restriction claim, we only need to prove (1) and (3) in the claim of Theorem \ref{thm:pull-back}. Let us first prove (3).
	
	Let $T\in M_k(Y)$ be a current of finite mass. For any compact set $K\subset X$, Theorem \ref{thm: locpull} (3) implies the estimate \[ L^{-k}\int_Yf_\#\chi_K\ud\|T\|\le \|f^\ast T\|(K)\le L^k \int_Yf_\#\chi_K\ud\|T\|. \] Let  $(K_j)$ be an increasing sequence of compact sets in $X$ for which $$\bigcup_{j\in\N}K_j=X.$$
	Then $$f_\#\chi_{K_j}\to f_\#\chi_X=\deg f$$ pointwise. Thus
	\begin{equation}\label{xz}
		\lim_{j\to\infty}\int_Yf_\#\chi_{K_j}\ud\|T\|= (\deg f)\int_Y\ud\|T\|=(\deg f)M(T).
	\end{equation}
	This proves (3).
	
	By (\ref{xz}) and Theorem \ref{thm: locpull}(1), we have \[ f_\ast(f^\ast T\lfloor {\chi_{K_j}})=T\lfloor {f_\#\chi_{K_j}}\to (\deg f)T \] weakly in $M_{k,\loc}(X)$ as $j\to\infty$. On the other hand \[ f^\ast T\lfloor {\chi_{K_j}}\to f^\ast T \] weakly in $M_{k,\loc}(X)$ as $j\to\infty$. Thus (1) in Corollary \ref{thm:pull-back} is proven.
	
	By Theorem \ref{thm: locpull} (2) and (3), $f_\#$ maps $N_k(Y)$ to $N_k(X)$.
\end{proof}




\section{Equidistribution estimates for pull-back currents} 
\label{sec:equidistribution}

\subsection{BLD-maps from $\R^n$ into metric spaces}\label{sec:BLDrn} Let $X$ be a metric space and $f:\R^n\to X$ a Lipschitz map. We will use the \emph{metric Jacobian} $Jf$ of $f$, defined by Kirhchheim \cite{kir94}: for almost every $x\in \R^n$ the limit 
\begin{equation}\label{kirchheim}
	\lim_{h\to 0}\frac{d(f(x+hv),f(x))}{|h|}, \ v\in \R^n
\end{equation}
exists for all $v\in\R^n$ and defines a seminorm. The \emph{metric differential} $${\rm md }_xf:\R^n\to [0,\infty)$$ of $f$ at such a point $x$ is the seminorm given by (\ref{kirchheim}) and zero otherwise. This induces the metric Jacobian $Jf:\R^n\to\R$, a Borel function defined for any point where the limit (\ref{kirchheim}) exists, by \[ Jf(x)=\left(\int_{S^{n-1}} {\rm md}_xf(v)^{-n}\ud \sigma_{n-1}(v)\right)\inv.  \] Here $\sigma_{n-1}$ is the normalized surface measure on the unit sphere $S^{n-1}$ of $\R^n$. The metric Jacobian plays a prominent role in the \emph{co-area formula}
\begin{equation}\label{coarea}
	\int_{\R^n}gJf\ud x=\int_X\left(\sum_{x\in f\inv(y)}g(x)\right)\ud\Ha^n(y).
\end{equation}
We refer to \cite{kir94} for details.

\begin{remark}\label{rem:bound}
	By \cite[Lemma 2.4]{lui17} we obtain that, if the limit in (\ref{kirchheim}) exists for $x\in \R^n$, then
	\[
	L\inv\le  {\rm md}_x(v)\le L
	\] for any $v\in S^{n-1}$. 
\end{remark}

Remark \ref{rem:bound} together with the co-area formula (\ref{coarea}) implies that a BLD-elliptic oriented cohomology $n$-manifold is locally Ahlfors $n$-regular; see \cite[Proposition 6.3 and Remark 4.16(b)]{hei02}.

\medskip\noindent Throughout the rest of this section $X$ is a compact geodesic oriented cohomology $n$-manifold, and $f:\R^n\to X$ an $L$-BLD map. We denote \[ |X|=\Ha^n(X) \textrm{ and } D=\diam(X). \] By Remark \ref{rem:bound} and the discussion after it the space $X$ is Ahlfors $n$-regular under the present assumptions. In particular $|X|\le CD^n$, where $C>0$ is the Ahlfors regularity constant.

\subsection{Equidistribution}

We turn our attention to the value distribution of BLD-maps. The following theorem will be used in the next subsection to obtain estimates on the mass of pullbacks of currents. For the theorem, let $A_f:(0,\infty)\to(0,\infty)$ be the function  $$R\mapsto\frac{1}{|X|}\int_{B(R)}Jf\ud x,$$
and denote $B(R)=B(0,R)\subset \R^n$.
\begin{theorem}\label{equidist}
	Let $f:\R^n\to X$ be an $L$-BLD-mapping to a compact geodesic oriented cohomology manifold. Then there exists a constant $c(n,L)>0$ for which
	\begin{equation}
		\left(1-\frac{c(n,L)D}{R}\right)\le \frac{f_\#\chi_{B(R)}(p)}{A_f(R)}\le\left(1+\frac{c(n,L)D}{R}\right)
	\end{equation}
	for every $p\in X$ and $R\ge LD$.
\end{theorem}

Theorem \ref{equidist} gives a quantitative equidistribution estimate with constants depending only on $n$ and $L$. We refer to \cite{mat79} and \cite{pan10} for similar results for quasiregular maps. We begin with an observation which we record as a lemma.

\begin{lemma}
	For a compactly supported Borel function $h:\R^n\to \R$, we have \[ \int_Xf_\# h \ud \Ha^n=\int_{\R^n} hJf\ud x.\]
	\begin{proof}
		By \cite[Corollary 10.2]{onn09} we have $|f^{-1}fB_f|=0$. Thus $$f_\# h(y)=\sum_{x\in f^{-1}(y)}h(x)$$ for almost every $y\in \R^n$. The rest follows directly from the change of variables formula (\ref{coarea}).
	\end{proof}
\end{lemma}

\begin{proof}[Proof of Theorem \ref{equidist}]
	Let $R>LD$, and let  $\chi_{B(R-\delta)}\le \eta \le \chi_{B(R)}$ be a Lipschitz function. By Lemma \ref{bijection}, for points $p,q\in X\setminus fB_f$, there is a bijection $$\psi: f^{-1}(q)\to f^{-1}(p)$$ satisfying $$ d(p,q)/L\le |x-\psi(x)|\le Ld(p,q)$$ for all $x\in f\inv(q)\cap B(R)$. Thus \[f_\ast \eta(p)-f_\ast \eta(q) \le \sum_{x\in f^{-1}(q)}|\eta(\psi(x))-\eta(x)|.\]
	
	We have $|\eta(\psi(x))-\eta(x)|\le 1$ for all $x$. Moreover, 
	\begin{itemize}
		\item[(i)] $|\eta(\psi(x))-\eta(x)|=0$, if $x\in B(R-LD-\delta)$, and 
		\item[(ii)] $|\eta(\psi(x))-\eta(x)|=0$, if $x\in \R^n\setminus B(R+LD)$.
	\end{itemize}
	Therefore \[ \sum_{x\in f^{-1}(q)}|\eta(\psi(x))-\eta(x)|\le  \sum_{x\in f^{-1}(q)}\chi_{B(R+LD)\setminus B(R-LD-\delta)}(x)\] and we obtain \[ f_\# \eta(p)\le f_\# \chi_{B(R)}(q)+f_\# \chi_{B(R+LD)\setminus B(R-LD-\delta)}(q)\] for almost every $q$. Denote $$A(t,s)=B(t)\setminus B(s),\quad \textrm{ for }t>s.$$
	
	Integrating with respect to $q$ we obtain 
	\begin{align}\label{1}
		f_\# \eta(p)\le & \dashint_Xf_\# (\chi_{B(R)}+\chi_{B(R+LD)\setminus B(R-LD-\delta)})\ud q \nonumber \\ 
		=&\frac{1}{|X|}\left(\int_{B(R)}Jf\ud x+\int_{A(R+LD,R-LD-\delta)}Jf\ud x\right)
	\end{align}
	In similar fashion we may obtain the estimate \[ f_\# \eta(p)-f_\# \eta(q)\le f_\# \chi_{A(R+LD,R-LD-\delta)}(p). \] Fixing $q$ and integrating with respect to $p$ yields \[\dashint_{X} f_\ast \eta\ud p \le f_\ast \eta(q)+ \dashint_{X}f_\# \chi_{A(R+LD,R-LD-\delta)}\ud p,\] or
	\begin{equation}\label{2}
		\frac{1}{|X|}\left(\int_{B(R-\delta)}Jf\ud x-\int_{A(R+LD, R-LD-\delta)}Jf\ud x\right)\le f_\#\eta(q).
	\end{equation}
	Since $f_\#\eta$ is continuous, estimates (\ref{1}) and (\ref{2}) hold for all $p\in Y$. Letting $\delta\to 0$ we obtain
	\begin{align*}
		1-\frac{1}{|X|A_f(R)}\int_{A(R+LD,R-LD)}Jf\le \frac{f_\# \chi_{B(R)}(p)}{A_f(R)}\le 1+\frac{1}{|X|A_f(R)}\int_{A(R+LD,R-LD)}Jf
	\end{align*}
	for all $p\in M$. 
	
	Furthermore \[ \frac{1}{|X|A_f(R)}\int_{A(R+LD,R-LD)}Jf\le \frac{L^n2LDn(R+LD)^{n-1}}{L^{-n}R^n}\le\frac{2^nnL^{2n+1}D}{R}. \] This implies the claim.
\end{proof}

\subsection{Mass and flat norm estimates}
We apply the equidistribution Theorem \ref{equidist} to prove estimates for the mass and flat norm of pull-backs of locally normal currents.

\begin{theorem}\label{estimate}
	Let $f:\R^n\to X$ be an $L$-BLD map into a compact, geodesic oriented cohomology $n$-manifold $X$. Let $T\in N_{k}(X)$ and $R>CD$, where $C$ is the constant in Theorem \ref{equidist}. Then there is a constant $c=c(n,L,k)$ for which
	\begin{align*}
		&\frac 1cA_f(R)\F(T)\le \F_{B_R}(f^\ast T)\le c A_f(R)\F(T)
	\end{align*}
	and
	\begin{align*}
		&\frac 1c A_f(R)M(T)\le \|f^\ast T\|(B(R))\le c A_f(R)M(T).
	\end{align*}
\end{theorem}

\begin{proof}
	Denote $\chi_{B(R)}=:\chi_R$. By Theorem \ref{equidist} we have \[ \frac 12\le \frac{f_\#\chi_R}{A_f(R)}\le 2 \] for $R>2C(n,L)D$. Thus
	\begin{align*}
		\frac 12L^{-k} A_f(R)M(T)&\le L^{-k}\int_X f_\#\chi_R\ud\|T\|=L^{-k}f^\ast\|T\|(B(R))\\
		&\le  \|f^\ast T\|(B(R))\le L^kf^{\ast}\|T\|(B(R))\\
		&=L^k\int_X f_\#\chi_R\ud\|T\|\le 2L^kA_f(R)M(T),
	\end{align*}
	establishing the first estimate.
	
	To estimate the flat norm, let $A\in N_{k+1}(X)$. Then, by Proposition \ref{massest}, we have that
	\begin{align*}
		\|f^\ast T-\partial f^\ast A\|(B_R)+&\|f^\ast A\|(B_R)\\
		=&\|f^\ast(T-\partial A)\|(B_R)+\|f^\ast A\|(B_R)\\
		\le & 2L^kA_f(R)M(T-\partial A)+2L^{k+1}A_f(R) M(A)\\
		\le & 2L^{k+1}A_f(R)( M(T-\partial A)+M(A) ).
	\end{align*}
	Thus \[ \F_{B(R)}(f^\ast T)\le 2L^{k+1}A_f(R)\F(T).  \]
	
	For the opposite inequality, let $\eta:\R^n\to [0,\infty)$, $\eta(x)= (1-\dist(B_{R-1},x))_+$ be a Lipschitz function. By the proof of Proposition \ref{thm: locpull} (2) we have \[ T\lfloor {f_\ast\eta}=f_\ast((f^\ast T)\lfloor \eta). \] Let $\displaystyle \varphi=1-\frac{f_\ast\eta}{A_f(R)}$.

	\medskip\noindent{\bf Claim.}
	We have 
	$$\|\varphi\|_\infty+\Lip\varphi\le C(n,L)(D+1)/R,$$
	where $c(n,L)$ is a constant depending only on $n$ and $L$.
	\begin{proof}[Proof of Claim] We observe first that 
		\[|\varphi| \le \left|1-\frac{f_\#\chi_{R}}{A_f(R)}\right|+\frac{f_\#\chi_{R}-f_\#\eta}{A_f(R)}\le \frac{cD}{R}+\frac{f_\#\chi_{R}-f_\#\chi_{R-1}}{A_f(R)}, \] and \[ \Lip\varphi=\frac{\Lip(f_\#\eta)}{A_f(R)}\le L\frac{f_\#(\chi_{R}-\chi_{R-1})}{A_f(R)}. \] Thus it suffices to estimate $\displaystyle f_\#(\chi_{R}-\chi_{R-1})/A_f(R)$. By Theorem \ref{equidist} we have that
		\begin{align*}
			\frac{f_\ast(\chi_R-\chi_{R-1})}{A_f(R)}&\le 1+\frac{cD}{R}-\frac{A_f(R-1)}{A_f(R)}(1-\frac{cD}{R-1})\le \frac{A_f(R)-A_f(R-1)}{A_f(R)}+\frac{2cD}{R}\\
			&\le\frac{nL^nR^{n-1}}{L^{-n}R^n}+\frac{2cD}{R}= \frac{nL^{2n}+cD}{R},
		\end{align*}
		for $R>CD$. 
	\end{proof}

	\medskip\noindent By Lemmas \ref{flatprod} and \ref{locpush}, we have
	\begin{align}\label{number}
		\F(T)&=\F\bigg(T\lfloor(f_\#\eta/A_f(R))+T\lfloor \varphi \bigg)\le \frac{1}{A_f(R)}\F(T\lfloor f_\#\eta)+\F(T\lfloor\varphi)\nonumber \\
		&\le \frac{L^{k+1}(\|\eta\|_\infty+\Lip\eta)}{A_f(R)}\F(f^\ast T)+(\|\varphi\|_\infty+\Lip\varphi)\F(T)\nonumber\\
		&\le \frac{2L^{k+1}}{A_f(R)}\F(f^\ast T)+(\|\varphi\|_\infty+\Lip\varphi)\F(T).
	\end{align}
	
	If $R>2C(n,L)(D+1)$, then (\ref{number}) yields the estimate \[ \F(T)\le \frac{c(n,L,k)}{A_f(R)}\F(f^\ast T)+\frac 12\F(T), \] from which the remaining inequality readily follows.
\end{proof}

\section{Homology of normal metric currents}
\label{sec:current_homology}

In this section we assume that $X$ is a compact oriented cohomology manifold and, in addition, that $X$ is \emph{locally Lipschitz contractible}. Recall that $X$ is locally Lipschitz contractible if every neighborhood $U$ of every point $x\in X$ contains a neighborhood $V\subset U$ of $x$ so that there is a Lipschitz map $$h:[0,1]\times V\to U$$ so that $h_1(y)=y$ for every $y\in V$ and $h_0$ is constant. We remark that this is similar to the notion of $\gamma$-Lipschitz contractibility in \cite[Section 3.2]{wen07}. For compact spaces it is not difficult to see that the two notions coincide in the sense that a locally Lipschitz contractible is $\gamma$-Lipschitz contractible for some $\gamma$, and a $\gamma$-Lipschitz contractible space is locally Lipschitz contractible.

\subsection{Current homology and oriented cohomology manifolds}\label{sec:homology}

The boundary map $$\partial_k :N_{k}(X)\to N_{k-1}(X)$$ satisfies $\partial_{k-1}\partial_k=0$, which can be readily seen from the definition of metric currents; see also \cite[Section 3]{lan11}. Thus the boundary map induces a chain complex
\begin{equation}\label{chain}
	\cdots \stackrel{\partial}{\longrightarrow} N_k(X)\stackrel{\partial}{\longrightarrow}N_{k-1}(X)\stackrel{\partial}{\longrightarrow}\cdots \stackrel{\partial}{\longrightarrow}N_0(X)\stackrel{}{\longrightarrow}0
\end{equation} 
As is customary we omit the subscripts from $\partial$. 

We study the homology of the chain complex (\ref{chain}) for a BLD-elliptic oriented cohomology manifold $X$ and we denote the homology groups of (\ref{chain}) by
\begin{equation}\label{homology}
	H_k(X):=\ker\partial_{k}/\ran\partial_{k+1},
\end{equation}
for $k\ge 0$. 

It is known that $H_\ast(\cdot)$ defines a homology theory satisfying the Eilenberg--Steenrod axioms; see \cite{mit15} and also \cite{wen07} for integral currents, and \cite{dep17} for the homology of normal chains and cohomology of charges. For us, \emph{homology} always refers to the homology (\ref{homology}) of (\ref{chain}).


\begin{remark}\label{rmk:homology}
	In what follows we compare the homology $H_\ast(X)$ with the singular homology $H^s_\ast(X;\R)$ with real coefficients. The assumption of local Lipschitz contractibility ensures that, on spaces that have the homotopy type of a CW-complex, the normal current homology $H_\ast(X)$ coincides with singular homology $H_\ast^s(X;\R)$; see \cite[Corollary 1.6]{mit15}.
\end{remark}

\subsection{Filling inequalities}
We say that a locally compact metric space $X$ \emph{admits a filling inequality for $N_k(X)$} if there is a constant $C>0$ such that each $T\in N_{k+1}(X)$ satisfies $$\FillVol(\partial T)\le CM(\partial T).$$ Recall that the \emph{filling volume} of a current $A\in N_k(X)$ is defined to be $$ \FillVol(A) =\inf\{ M(B): \ \partial B=A \},$$ the infimum over the empty set being understood as infinity. This means in particular that, if $S\in N_k(X)$ and $S=\partial T'$ for some $T'\in N_{k+1}(X)$, then there exists $T\in N_{k+1}(X)$ satisfying $S=\partial T$ and 
\begin{equation}\label{filling}
	M(T)\le CM(S).
\end{equation}

There is a related notion of \emph{cone type inequalities} introduced by Wenger \cite{wen05}. A space $X$ is said to \emph{support cone type inequalities for $N_k(X)$} if there exists a constant $C>0$ with the property that, if $S\in\ker\partial_k$, then there exists $T\in N_{k+1}(X)$ satisfying $\partial T=S$ and \[ M(T)\le C\diam(\spt S)M(S). \] A space $X$ supporting a cone type inequality for $N_k(X)$ necessarily has trivial current homology $H_k(X)$, whereas spaces admitting filling inequalities only require (\ref{filling}) for currents $S$ \emph{a priori known to have a filling}.

\begin{remark}
	In \cite{dep17}, De Pauw, Hardt, and Pfeffer introduce the notion of \emph{locally acyclic spaces}, see \cite[Definition 16.10]{dep17}. Locally Lipschitz contractible spaces are locally acyclic spaces, but the connection between filling inequalities and local acyclicity is not clear to us. 
\end{remark}

In this subsection we prove that compact BLD-elliptic spaces as in Theorem \ref{thm: main} support filling inequalities.

\begin{proposition}\label{main2}
	Let $f:\R^n\to X$ be an $L$-BLD map into a compact, geodesic, oriented and locally Lipschitz contractible cohomology $n$-manifold $X$, and let $0\le k\le n$. Then there exists a  constant $C>0$ having the property that, for every $T\in\ran \partial_{k+1}$ there exists $S\in N_{k+1}(X)$ satisfying $\partial S=T$ and \[ M(S)\le CM(T). \]
\end{proposition}


Filling inequalities are equivalent to the closedness of the range of $\partial$. We show this using finite dimensionality of the homology.

\begin{lemma}\label{cwtype}
	Let $X$ be a compact locally Lipschitz contractible metric space of finite covering dimension. Then the current homology $H_k(X)$ is finite dimensional for all $k\in\N$. 
\end{lemma}

\begin{proof}
	By \cite[Theorem V.7.1]{hu65} the space $X$ is an Euclidean neighborhood retract and by \cite[Corollary A.8]{hat02} it has the homotopy type of a finite CW-complex. By \cite[Corollary 1.6]{mit15} the normal current homology groups are isomorphic to the  singular homology groups (with real coefficients), and thus finite dimensional. 
\end{proof}

Lemma \ref{cwtype} immediately yields the desired finite dimensionality as a corollary. 
\begin{corollary}\label{finite}
	Let $X$ be a compact, locally geodesic,  orientable, and locally Lipschitz contractible cohomology $n$-manifold. Then the normal current homology groups $H_k(X)$ are finite dimensional for all $k\in\N$.
\end{corollary}

\begin{lemma}\label{closed}
	Let $X$ be a compact, locally geodesic,  orientable, and locally Lipschitz contractible cohomology $n$-manifold, and $k\ge 1$. Then the boundary operator $$\partial=\partial_k:N_k(X)\to N_{k-1}(X)$$ has closed range.
\end{lemma}

\begin{proof}
	Since $\ran\partial\subset \ker\partial_{k-1}$, we may consider $\partial$ as an operator $$\partial: N_k(X)\to \ker\partial_{k-1}.$$ By Corollary \ref{finite} the subspace $\ran\partial$ has finite co-dimension in $\ker\partial_{k-1}$. Then $\ran\partial$ is closed in $\ker\partial_{k-1}$ and thus in $N_{k-1}(X)$; see e.g. \cite[Corollary 2.17]{abr02}.
\end{proof}


We are now ready for the proof of the filling inequality.

\begin{proof}[Proof of Proposition \ref{main2}]
	Let $k\ge 0 $ and consider the operator $\partial =\partial_{k+1}$. By Lemma \ref{closed}, $(\ran\partial,N)$ is a Banach space. The canonical operator $$\overline\partial:N_{k+1}(X)/\ker\partial\to\ran\partial$$ is injective and onto. By the open mapping theorem, there is a constant $0<c<\infty$ for which
	\begin{equation}\label{below}
		M(\partial T)=N(\overline\partial [T])\ge c\|[T]\|_{N_{k+1}(X)/\ker\partial}=c\inf\{ N(T-A):\partial A=0 \}
	\end{equation}
	for every $T\in N_k(X)$. Let $A\in \ker_{k+1}\partial$. Then $$ \FillVol(\partial T)\le M(T-A)\le N(T-A).$$ This implies \[ \FillVol(\partial T)\le c\inv M(\partial T), \] and consequently the filling inequality for $N_k(X)$. 
\end{proof}

\subsection{Homological boundedness}

We use the filling inequality to establish the existence of mass minimal elements in homology classes of $H_\ast(X)$.

\begin{lemma}\label{mini}
	Let $X$ be compact, geodesic and locally Lipschitz contractible oriented cohomology manifold, and let $k\ge 0$ be an integer. Then each homology class $[S]=S+Im \partial\in H_k(X)$ contains an element $T\in[S]$ minimizing the flat norm $\F$ in $[S]$. Moreover, $T$ satisfies \[ \F(T)=M(T) \] and minimizes $\F$ in $[S]$ as well.
\end{lemma}

\begin{proof}
	Suppose $S_m=S+\partial A_m$ is a minimizing sequence in $[S]$, and denote  $B=\sup_mM(S_m)<\infty$. By Proposition \ref{main2}, we may assume that \[ M(A_m)\le CM(\partial A_m)=CM(S-S_m)\le C(M(S)+B) \] for each $m\in\N$, where $C$ is the constant in the claim of Proposition \ref{main2}. Thus $$\sup_mN(A_m)<\infty.$$ By passing to a subsequence  we may assume that the sequence  $(A_m)$ converges weakly to a normal current $A\in N_{k+1}(X)$. By the lower semicontinuity of the mass, \[ M(S+\partial A)\le \liminf_{m} M(S+\partial A_m)=\inf\{M(T): T\in [A] \}. \] Thus $S+\partial A$ is a mass minimizer in $[S]$.
	
	Let $T\in[S]$ be a mass minimizer in $[S]$. Then the inequality $\F(T)\le M(T)$ holds automatically. Further, for any $A\in N_{k+1}(X)$, 
	\[ 
	M(T)\le M(T-\partial A)\le M(T-\partial A)+M(A). 
	\] 
	Thus, taking infimum over $A\in  N_{k+1}(X)$ yields $M(T)\le \F(T)$.

	The equality $\F(T)=M(T)$ implies that, for any $A\in S+\ran\partial$ and $B\in N_{k+1}(X)$, \[ \F(T)=M(T)\le M(A-\partial B)\le M(A-\partial B)+M(B). \] Taking infimum over $B$ proves the last claim.
\end{proof}

\section{Proof of a non-smooth Bonk--Heinonen theorem}

To prove Theorem \ref{thm: main} we introduce a norm $|\cdot|:H_k(X)\to [0,\infty)$ on the homology group $H_k(X)$ by \[ c\mapsto\inf\{ M(T):T\in c \}. \] By Lemma \ref{mini} each homology class $c\in H_k(X)$ contains an element of minimal norm, and in particular $|c|>0$ if and only if $c\ne 0$.

\begin{proof}[Proof of Theorem \ref{thm: main}]
	By scaling the map and the metric of the space $X$ we may assume $D=\diam X=1$. Let $m=\dim H_k(X)$. Then there exists linearly independent homology classes $[T_1],\ldots [T_m]\in H_k(X)$ satisfying
	\[ |[T_i]|=M(T_i)=1\textrm{ and } |[T_i]-[T_j]|\ge 1/2\textrm{ if }i\ne j,  \] for $i,j\in \{1,\ldots,m\}$. Let $S_{ij}$ be mass minimizers in the homology class $[T_i-T_j]$. By Lemma \ref{mini}, we have \[ 1/2\le |[T_i-T_j]|=M(S_{ij})=\F(S_{ij})\le \F(T_i-T_j) \] for $i\ne j$. Let $R=2C$, where $C$ is the constant in Theorem \ref{equidist}. By Theorem \ref{estimate}, there are constants $a=a(n,L)>0$ and  $b=b(n,L)>0$, depending only on $n$ and $L$, for which 
	\begin{align}\label{b1}
		\F_{B(R)}(f^\ast T_i-f^\ast T_j)\gtrsim_{n,L}A_f(R)\F(T_i-T_j)\ge a(n,L)\frac{R^n}{|X|}
	\end{align}
	and 
	\begin{align}\label{b2}
		\|f^\ast T_i\|(B_R)\simeq_{n,L} A_f(R)\ge b(n,L)\frac{R^n}{|X|}
	\end{align}
	for all $i,j\in \{1,\ldots,m\},\ i\ne j$.
	
	Let $\eta:\R^n\to\R$ be the Lipschitz function $x\mapsto (1-\dist(B(R),x))_+$, and define $\displaystyle S_i:=|X|(f^\ast T_i)\lfloor\eta$ for each $i=1,\ldots,m$. Note that $\spt\eta\subset B(R+1)=B(2C+1)$. Thus $S_i$ is supported in $B(2C+1)$ for each $i=1,\ldots,m$. By (\ref{b2}) 
	\[ 
	M(S_i)=N(S_i)\le b(2C+1)^n  
	\]
	for each $i=1,\ldots,m.$ Moreover, by (\ref{b1}) we have \[ \F_{B(2C+1)}(S_i-S_j)\ge \F_{B(2C)}(S_i-S_j)\ge  a(2C)^n\] whenever $i\ne j$; cf Lemma \ref{flatprod}.
	
	Thus
	\begin{equation}\label{compe}
		\{ S_1,\ldots,S_m \}\subset N(\bar B(2C+1),b(2C+1)^n),
	\end{equation} where, for an compact set $K\subset \R^n$, and $\lambda\ge 0$, $$N(K,\lambda)=\{T\in N_k(\R^n):\spt T\subset U,\ N(T)\le \lambda \}. $$ By Theorem \ref{compactness} the right-hand side in (\ref{compe}) is compact in $\F_{B(2C+1)}$. Therefore there is an upper bound $m(n,L)$, depending only on $n$ and $L$, on the cardinality of a finite set $\mathcal S$ in $N(B(2C+1),b(2C+1)^n)$ having the property that \[ \F_{B(2C+1)}(S-S')\ge b(2C)^n \] whenever $S,S'\in \mathcal S$ and $S\neq S'$. We conclude that $m=\dim H_k(X)\le m(n,L)$. The proof is complete.
\end{proof}

\appendix

\section{Local Euclidean bilipschitz embeddability of BLD-elliptic spaces}\label{appendix}

In this appendix we prove the following embeddability theorem mentioned in the introduction.

\begin{theorem}\label{lem:almgren}
	Let $X$ be a locally geodesic, orientable cohomology manifold admitting a BLD-map $f:\R^n\to X$. Let $x\in X$. For every radius $r>0$, for which there exists $y\in f\inv(x)$ such that $U(y,r)$ is a normal neighborhood of $y$, $B_r(x)$ is bilipschitz equivalent to a subset of a Euclidean space.
\end{theorem}

In the proof we use Almgren's theory of $Q$-valued maps. We refer to \cite{del11} for a recent exposition. Denote by $\mathcal A_Q(\R^n)$ the space of \emph{unordered} $Q$-tuples of points in $\R^n$. For the purpose of introducing a metric, we formally define 
\[
\mathcal A_Q(\R^n)=\left\{ \sum_{i=1}^Q\delta_{x_i}: x_1,\ldots,x_Q\in\R^n \right\},
\]
where $\delta_x$ is the Dirac mass at $x\in\R^n$. Given $T_1,T_2\in \mathcal A_Q(\R^n)$, suppose
\[
T_1=\sum_{i=1}^Q\delta_{x_i},\quad T_2=\sum_{i=1}^Q\delta_{y_i},
\]
and define
\[
d_Q(T_1,T_2)=\min\left\{\left(\sum_{i=1}^Q|x_i-y_{\sigma(i)}|^2\right)^{1/2}: \sigma\in S_Q \right\},
\]
where $S_Q$ denotes the set of permutations of $\{1,\ldots, Q\}$. A key property of $\mathcal A_Q(\R^n)$ is the following bilipschitz embedding result.

\begin{theorem}\cite[Theorem 2.1]{del11} \label{bilipemb}
	There exists $N=N(Q,n)$ and a bilipschitz map $$\xi:\mathcal A_Q(\R^n)\to \R^N.$$
\end{theorem}

Let  $x\in X$ and $r>0$, and suppose that $y\in f\inv(x)$ has the property that $U=:U(y,r)$ is a normal neighborhood of $y$. Set $Q=i_f(y)$. We define a $Q$-valued map $g_f:B_r(x)\to \mathcal A_Q(\R^n)$ by
\begin{equation}\label{qvalued}
	x\mapsto \sum_{z\in f\inv(x)\cap U}i_f(z)\delta_z.
\end{equation}

\begin{lemma}
	Let $f:X\to Y$ be a BLD-map between locally geodesic, oriented cohomology manifolds. Then, the map $g:B_r(x)\to \mathcal A_Q(\R^n)$ is a bilipschitz embedding.
\end{lemma}

\begin{proof}
	Let $p,q\in B_r(x)\setminus fB_f$. By Lemma \ref{bijection}  and its proof, there is a bijection $$\psi:f\inv(p)\cap U\to f\inv(q)\cap U$$ satisfying $$d(p,q)/L\le d(z,\psi(z))\le Ld(p,q)$$ for each $x\in f\inv(p)\cap U$. Thus
	\[
	d_Q(g_f(p),g_f(q))\le \left(\sum_{z\in f\inv(p)\cap U}d(z,\psi(z))^2\right)^{1/2}\le L\sqrt Q d(p,q).
	\]
	For the opposite inequality let $f\inv(p)\cap U=\{x_1,\ldots,x_Q \}$, $f\inv(q)\cap U=\{ y_1,\ldots,y_Q \}$, and $\sigma\in S_Q$. Then, for each $i=1,\ldots,Q$,
	\[
	|x_i-y_{\sigma(i)}|\ge \frac 1L \ell(f\circ [x_i,y_{\sigma(i)}])\ge \frac  1L d(p,q),
	\]
	where $[x_i,y_{\sigma(i)}]$ denotes the geodesic line segment from $x_i$ to $y_{\sigma(i)}$. Thus \[ d_Q(g_f(p),g_f(q))\ge \sqrt Qd(p,q)/L. \]
	
	We have established the bilipschitz condition for points $p,q$ in the dense set $B_r(x)\setminus fB_f$, whence it follows for all $p,q\in B_r(x)$.
\end{proof}
\begin{proof}[Proof of Theorem \ref{lem:almgren}]
	Let $x\in X$ and let $r>0$ be a radius with the property that there exists $y\in f\inv(x)$ for which $U=U(y,r)$ is a normal neighborhood of $y$. Set $Q=i_f(y)$ and consider the map $g_f:B_r(x)\to \mathcal A_Q(\R^n)$. Then the map $$\xi\circ g:B_r(x)\to \R^{N},$$ where $\xi:\mathcal A_Q(\R^n)\to \R^N$ is the map of Theorem \ref{bilipemb}, is bilipschitz.
\end{proof}

\begin{remark}
	Since the index $i_f$ of a BLD-map $f:\R^n\to X$ is bounded by a constant depending only on $n$ and $L$, the bilipschitz constant of $\xi\circ g$ and the dimension $N$ are also bounded by constants depending only on $n$ and $L$. 
\end{remark}

\bibliographystyle{plain}
\bibliography{abib}

\end{document}